\documentclass[12pt]{article}
\usepackage{amsfonts, amsmath, amssymb, latexsym, eucal, amscd} 
\usepackage{graphicx,lscape} 

\usepackage{amsthm}
\usepackage{amscd}

\usepackage{cite}

\newtheorem{definition}{Definition}[section]
\newtheorem{theorem}[definition]{Theorem}
\newtheorem{lemma}[definition]{Lemma}
\newtheorem{corollary}[definition]{Corollary}
\newtheorem{remark}[definition]{Remark}

\newtheorem{conjecture}[definition]{Conjecture}
\newtheorem{problem}[definition]{Problem}

\newtheorem{proposition}[definition]{Proposition}

\begin{document} 

\title{\bf Strengthening the balanced set condition \\
for the distance-regular graph \\
of the bilinear forms
}
\author{
Paul Terwilliger \\
Department of Mathematics \\
University of Wisconsin \\
480 Lincoln Drive \\
Madison, WI 53706-1388 USA \\
Email: {\tt terwilli@math.wisc.edu }\\
\\
Jason Williford \\
Department of Mathematics and Statistics \\
University of Wyoming \\
1000 E. University Ave. \\
 Laramie, WY 82071  USA \\
Email: {\tt jwillif1@uwyo.edu}
}

\date{}

\maketitle
\begin{abstract}  We consider a distance-regular graph $\Gamma=(X, \mathcal R)$ called the bilinear forms graph $H_q(D,N-D)$; we assume $N>2D\geq 6$ and $q \not=2$.
We show that $\Gamma$ satisfies the following strengthened version of the balanced set condition.
 For a vertex $x \in X$ and $0 \leq i \leq D$ define $\Gamma_i(x)=\lbrace y \in X\vert \partial(x,y)=i\rbrace$, where $\partial$ denotes the path-length distance function.
 Abbreviate $\Gamma(x)=\Gamma_1(x)$.
 Let $V={\mathbb R}^X$ denote the standard module for  ${\rm Mat}_X(\mathbb R)$.
 For $x\in X$ let $\hat x \in V$  have $x$-coordinate 1 and all other coordinates 0. 
  Let $E \in {\rm Mat}_X(\mathbb R)$ denote the primitive idempotent that corresponds to the second largest eigenvalue of the adjacency matrix of $\Gamma$.
 For a subset $\Omega \subseteq X$ define $\widehat \Omega = \sum_{x \in \Omega} \hat x$.
 We fix two vertices $x,y \in X$ and write $k=\partial(x,y)$. To avoid degenerate situations, we assume
 $2 \leq k \leq D-1$. Using $y$ we obtain an equitable partition $\lbrace O_i \rbrace_{i=1}^6$ of the local graph $\Gamma(x)$. 
 By construction
$O_1 = \Gamma (x) \cap \Gamma_{k-1}(y)$ and $O_6 = \Gamma(x) \cap \Gamma_{k+1}(y)$.
 We call  $\lbrace O_i \rbrace_{i=1}^6$  the $y$-partition of $\Gamma(x)$.
Let  $\lbrace O'_i \rbrace_{i=1}^6$ denote the $x$-partition of $\Gamma(y)$.
 According to the original balanced set condition, for $i \in \lbrace 1,6\rbrace$ the vector
$ E \widehat O_i - E \widehat O'_i$ is a scalar multiple of $E{\hat x}-E{\hat y}$. 
We show that for $1 \leq i \leq 6$ the vector
$ E \widehat O_i - E \widehat O'_i$ is a scalar multiple of $E{\hat x}-E{\hat y}$. 
We investigate the consequences of this result.
\medskip

 \end{abstract}
 
 \section{Introduction}
 Let $\Gamma=(X, \mathcal R)$ denote a distance-regular graph with diameter $D\geq 3$ (formal definitions begin in Section 2).
 In  \cite{QPchar} the first author introduced a linear algebraic condition on $\Gamma$ that is now called the balanced set condition (BSC).
 The original purpose of the BSC was to characterize the $Q$-polynomial property \cite[Theorem~1.1]{QPchar}.
 Since that beginning, the BSC has found many applications as we now review. Throughout this review, we assume that $\Gamma$
 is $Q$-polynomial.
 In \cite{ex1}, the BSC is used to determine when $\Gamma$ has an antipodal distance-regular cover.
In \cite{newIneq}, the BSC is used to obtain some equations involving the intersection numbers of $\Gamma$.
In \cite{ex3}, the BSC is used to show that $\Gamma$  has no kites under the assumption that $\Gamma$ has negative type.
In \cite{ex4}, the BSC is used to show that  $\Gamma$ has girth at most 6 provided that $\Gamma$ has valency at least 3.
 In \cite{ex5}, the BSC is used to obtain a product formula for the cosines of $\Gamma$ under the assumption that $\Gamma$ is tight.
 In \cite{caughman1}, the BSC is used to show that for $\Gamma$ bipartite  the last subconstituent of $\Gamma$  supports another $Q$-polynomial distance-regular graph.
 In \cite{CN2H}, the BSC is used to show that  $\Gamma$ is 1-homogeneous under the assumption that $\Gamma$ supports a spin model.
 In \cite{mik1,mik2, mik4} the BSC is used to obtain an equitable partition of $X$, under various assumptions on the intersection numbers of $\Gamma$.
 In \cite{ex11}, the BSC is used to obtain bounds on the triple intersection numbers of $\Gamma$.
In \cite{ex12}, the BSC is used to classify the thick $Q$-polynomial regular near $2D$-gons.
In \cite{ex13}, the BSC is used to express the $Q$-polynomial property of $\Gamma$ in terms of the intersection numbers alone.
In \cite{ex14}, the BSC is used to obtain a connectivity result concerning the last two subconstituents of $\Gamma$.
In \cite{norton2}, the BSC is used to describe the Norton algebra product for $\Gamma$ in a symmetric way.
 In \cite{nortonBalanced}, a variation on the BSC is given that involves the Norton algebra product.
 \medskip
 
 \noindent In the present paper, we assume that $\Gamma$ is the bilinear forms graph $H_q(D,N-D)$ with $N>2D\geq 6$ and $q \not=2$ \cite[p.~280]{bcn}.
    Let $A\in {\rm Mat}_X(\mathbb R)$ denote the adjacency matrix of $\Gamma$.
     It is known \cite[p.~357]{bbit} that
  $\Gamma$ is $Q$-polynomial with respect to the natural ordering $\theta_0 > \theta_1 > \cdots > \theta_D$ of the eigenvalues of $A$.
  In our main results, we describe how $\Gamma$ satisfies a strengthened version of the BSC. These results will be summarized shortly.
  \medskip
  
  \noindent Before we get into the details, we would like to acknowledge that this paper is motivated by the recent work of Ian Seong \cite{seong} concerning
  the Grassmann graphs.
  \medskip
  
  \noindent We have some comments about notation.
  Let $\partial$ denote the path-length distance function for $\Gamma$. For $x \in X$ and $0 \leq i \leq D$ define the set $\Gamma_i(x)=\lbrace y \in X\vert \partial(x,y)=i\rbrace$.
  We abbreviate $\Gamma(x)=\Gamma_1(x)$.
 Let $V={\mathbb R}^X$ denote the standard module for  ${\rm Mat}_X(\mathbb R)$.
 For $x\in X$ let $\hat x \in V$  have $x$-coordinate 1 and all other coordinates 0. By construction, the vectors $\lbrace {\hat x} \vert x \in X\rbrace$
 form a basis for $V$.
 We endow $V$ with a symmetric bilinear form with respect to which the basis $\lbrace \hat x \vert x \in X\rbrace$ is orthonormal.
  Let $E \in {\rm Mat}_X(\mathbb R)$ denote the primitive idempotent of $A$ that corresponds to $\theta_1$.
 Note that $EV={\rm Span} \lbrace E{\hat x} \vert x \in X\rbrace$
 is the $\theta_1$-eigenspace of $A$. We will show that for $x \in X$ the vectors $\lbrace E{\hat y} \vert y \in \Gamma(x) \rbrace$ form a basis for $EV$.
 For a subset $\Omega \subseteq X$ define $\widehat \Omega = \sum_{x \in \Omega} \hat x$.
 \medskip
 
 \noindent 
For the rest of this section,
 we fix two vertices $x,y \in X$ and write $k=\partial(x,y)$. To avoid degenerate situations, we always assume
 $2 \leq k \leq D-1$. Using $y$ we obtain an equitable partition $\lbrace O_i \rbrace_{i=1}^6$ of the local graph $\Gamma(x)$. 
 By construction
 \begin{align*}
 O_1 = \Gamma (x) \cap \Gamma_{k-1}(y), \qquad \qquad O_6 = \Gamma(x) \cap \Gamma_{k+1}(y)
 \end{align*}
 and $\lbrace O_i \rbrace_{i=2}^5$ partition $\Gamma(x) \cap \Gamma_k(y)$. Following \cite[Section~6]{seong} we call  $\lbrace O_i \rbrace_{i=1}^6$  the $y$-partition of $\Gamma(x)$.
Let  $\lbrace O'_i \rbrace_{i=1}^6$ denote the $x$-partition of $\Gamma(y)$.
 According to the original BSC, for $i \in \lbrace 1,6\rbrace$ the vector
$ E \widehat O_i - E \widehat O'_i$ is a scalar multiple of $E{\hat x}-E{\hat y}$. 
We show that for $1 \leq i \leq 6$ the vector
$ E \widehat O_i - E \widehat O'_i$ is a scalar multiple of $E{\hat x}-E{\hat y}$. 
\medskip

\noindent 
 We show that $\lbrace E\widehat O_i \rbrace_{i=1}^6$ and $\lbrace E\widehat O'_i \rbrace_{i=1}^6$ 
 form bases for the same subspace of $EV$; we denote this subspace by $S$.  We show that $E{\hat x}, E{\hat y} \in S$.
 We decompose $S$ into an orthogonal direct sum of two subspaces, called the symmetric part ${\rm Sym}(S)$ and antisymmetric part ${\rm ASym}(S)$.
 We do this as follows.
  For $1 \leq i \leq 6$  define $O^\vee_i = {\widehat O}_i - \lambda_i {\hat x}$, where $\lambda_i \in \mathbb R$ satisfies
  $E\widehat O_i - E \widehat O'_i=\lambda_i(E{\hat x}-E{\hat y})$.
 By construction,
  \begin{align*}
 E O^\vee_i = E \widehat O_i - \lambda_i  E \hat x = E \widehat O'_i - \lambda_i  E \hat y.
 \end{align*}
We show that $0 =\sum_{i=1}^6 EO^\vee_i$.
Define 
\begin{align*}
{\rm Sym}(S) &= {\rm Span}\lbrace EO^\vee_i \vert 1 \leq i \leq 6\rbrace, \\
 {\rm ASym}(S) &= {\rm Span}\lbrace E{\hat x}- E{\hat y}\rbrace.
 \end{align*} 
 We show that
  \begin{align*}
  S = {\rm Sym} (S) + {\rm ASym}(S) \qquad \quad \hbox{\rm (orthogonal direct sum).}
  \end{align*}
  
 \noindent We display two orthogonal bases for $S$, denoted $\lbrace h_j \rbrace_{j=1}^6$ and  $\lbrace h'_j \rbrace_{j=1}^6$.
 We show that for $1 \leq j \leq 6$ the vector $h_j - h'_j$ is a scalar multiple of  $E{\hat x}-E{\hat y}$.
 For $1 \leq j \leq 6$ we define the vector
 \begin{align*}
 h^\vee_j = h_j - \mu_j E{\hat x} = h'_j - \mu_j E{\hat y},
 \end{align*}
 where $\mu_j \in \mathbb R$ satisfies $h_j - h'_j = \mu_j (E{\hat x} - E{\hat y})$.
 We show that $h^\vee_1=0$ and  $h^\vee_6=h_6 = h'_6$.
 We show that $\lbrace h^\vee_j \rbrace_{j=2}^6$ is a basis for ${\rm Sym}(S)$. 
 \medskip
 
\noindent  Next, we bring in the Norton algebra product $\star$ on $EV$.
 We show that
 \begin{align*}
 E{\hat x} \star S \subseteq S, \qquad \qquad E{\hat y} \star S \subseteq S.
 \end{align*}
 We show that
   \begin{align*}
 {\rm ASym}(S) \star {\rm ASym}(S) &\subseteq {\rm Sym}(S), \\
 {\rm Sym}(S) \star {\rm ASym}(S) &\subseteq {\rm ASym}(S), \\
 (E{\hat x} + E{\hat y} ) \star {\rm Sym}(S) &\subseteq {\rm Sym}(S).
 \end{align*}
We  show that for $1 \leq j \leq 6$,
 \begin{align*}
 E{\hat x} \star h_j \in {\rm Span}\lbrace h_j\rbrace, \qquad \qquad  E{\hat y} \star h'_j \in {\rm Span}\lbrace h'_j\rbrace.
 \end{align*}
 
 \noindent
We previously mentioned $h^\vee_6=h_6 = h'_6$; we denote this common value by $\omega$.
By construction $\omega \in {\rm Sym}(S)$.
We show that
 \begin{align*}
  E\hat x \star \omega = -q \vert X \vert^{-1} \omega, \qquad \qquad
  E \hat y \star \omega = - q \vert X \vert^{-1} \omega.
  \end{align*}
  Let $\omega^\perp$ denote the orthogonal complement of $\omega$ in $S$. By construction,
   \begin{align*}
    S={\rm Span}\lbrace \omega\rbrace + \omega^\perp \qquad \quad \hbox{\rm (orthogonal direct sum).}
  \end{align*}
We obtain
$E{\hat x}, E{\hat y}  \in \omega^\perp$
and 
\begin{align*}
E{\hat x} \star \omega^\perp \subseteq \omega^\perp, \qquad \qquad E{\hat y} \star \omega^\perp \subseteq \omega^\perp.
\end{align*}
We show that $\omega^\perp$ has a basis
\begin{align*}
&E{\hat y}, \qquad \quad E{\hat x} \star E{\hat y}, \qquad \quad E{\hat x} \star ( E{\hat x} \star E{\hat y}), \qquad \quad  E{\hat x} \star (E{\hat x} \star ( E{\hat x} \star E{\hat y})), \\
& E{\hat x} \star ( E{\hat x} \star (E{\hat x} \star ( E{\hat x} \star E{\hat y})))
\end{align*}
and another basis 
\begin{align*}
&E{\hat x}, \qquad \quad E{\hat y} \star E{\hat x}, \qquad \quad E{\hat y} \star ( E{\hat y} \star E{\hat x}), \qquad \quad  E{\hat y} \star (E{\hat y} \star ( E{\hat y} \star E{\hat x})), \\
& E{\hat y} \star ( E{\hat y} \star (E{\hat y} \star ( E{\hat y} \star E{\hat x}))).
\end{align*}
We show that $\omega^\perp$ is spanned by the vectors 
\begin{align*}
& E{\hat z}_1\star ( E{\hat z_2} \star (E{\hat z_3} \star \cdots \star ( E{\hat z_{n-1}} \star E{\hat z_n} ) \cdots )), \\
&  n\geq 1, \qquad \quad z_i \in \lbrace x,y \rbrace \quad (1 \leq i \leq n).
\end{align*}
Using the above results, we express our strengthened BSC in the following way.
\begin{theorem} \label{thm:main2Intro} For an integer $n\geq 1$ and $1 \leq i \leq n$ let $z_i, {\sf z}_i$ denote a permutation of
$x, y$. 
Then
\begin{align*}
& E{\hat z}_1\star ( E{\hat z_2} \star (E{\hat z_3} \star \cdots \star ( E{\hat z_{n-1}} \star E{\hat z_n} ) \cdots )) \\
&-
E{\hat {\sf z}}_1\star ( E{\hat {\sf z}_2} \star (E{\hat {\sf z}_3} \star \cdots \star ( E{\hat {\sf z}_{n-1}} \star E{\hat {\sf z}_n} ) \cdots )) \\
& \in {\rm Span} \lbrace E{\hat x} - E{\hat y}\rbrace.
\end{align*}
\end{theorem}
\noindent We call the condition in Theorem \ref{thm:main2Intro} the bbalanced set condition.
 \medskip
 
 \noindent At the end of the paper, we mention an open problem concerning the bbalanced set condition.
 \medskip
 
 \noindent The paper is organized as follows. Section 2 contains some preliminaries.
 In Section 3, we recall some basic definitions and facts about a distance-regular graph $\Gamma=(X, \mathcal R)$.
 In Section 4, we assume that $\Gamma$ is a bilinear forms graph, and describe some features of $\Gamma$.
 In Section 5, for a vertex $x \in X$ we describe the local graph $\Gamma(x)$.
 In Section 6, for $x,y \in X$ at distance $2 \leq \partial(x,y) \leq D-1$ we introduce the $y$-partition of $\Gamma(x)$.
 In Section 7, we use the $y$-partition of $\Gamma(x)$ to define a subspace $S$, that we describe in various ways.
 In Section 8, we show that $\Gamma$ satisfies a strengthened version of the BSC.
 In Section 9, we discuss  two orthogonal bases for $S$.
 In Section 10, we express $S$ as an orthogonal direct sum of  ${\rm Sym}(S)$ and  ${\rm ASym}(S)$.
 In Section 11, we bring in the Norton algebra. We investigate $S$, ${\rm Sym}(S)$, ${\rm ASym}(S)$  from the Norton algebra point of view.
 In Section 12, we introduce the vector $\omega$ and investigate it from a Norton algebra point of view.
 In Section 13, we give some directions for future research.
 \medskip
 
 \noindent The main results of the paper are Theorem  \ref{thm:bbalanced} and Theorem \ref{thm:main2}.

 \section{Preliminaries} 
In this section, we review some basic concepts and notation that will be used throughout the paper.
Recall the natural numbers $\mathbb N = \lbrace 0,1,2,\ldots\rbrace$ and integers $\mathbb Z = \lbrace 0, \pm 1, \pm 2,\ldots \rbrace$.
Let $\mathbb R$ denote the field of real numbers. Let $X$ denote a finite set with cardinality $\vert X \vert \geq 2$.
An element of $X$ is called a {\it vertex}. 
Let ${\rm Mat}_X(\mathbb R)$ denote the $\mathbb R$-algebra consisting of the matrices that have rows and columns indexed by $X$ 
and all entries in $\mathbb R$.
Let $V=\mathbb R^X$ denote the $\mathbb R$-vector space consisting of the column vectors that have coordinates indexed by $X$ and all entries in $\mathbb R$.
The algebra ${\rm Mat}_X(\mathbb R)$ acts on $V$ by left multiplication. We call $V$ the {\it standard module}.
We endow $V$ with a symmetric bilinear form $\langle \,,\,\rangle $ such that $\langle u,v \rangle =u^t v$ for $u,v \in V$ ($t$ denotes transpose). We abbreviate
$\Vert v \Vert^2 = \langle v,v \rangle $ for all $v \in V$. Note that $\langle Bu,v \rangle = \langle u, B^t v \rangle $ for all $B \in {\rm Mat}_X(\mathbb R)$ and $u,v \in V$.
For $x \in X$ let $\hat x$ denote the vector in $V$ that has $x$-coordinate $1$ and all other coordinates $0$. The vectors $\lbrace \hat x \vert x \in X\rbrace$
form an orthonormal basis for $V$.

\section{Distance-regular graphs}
In this section, we recall some definitions and basic facts related to distance-regular graphs.
Let $\Gamma=(X, \mathcal R)$ denote an undirected, connected graph, without loops or multiple edges, with vertex set $X$,
adjacency relation $\mathcal R$, and path-length distance function $\partial$.
By the {\it diameter} of $\Gamma$ we mean 
\begin{align*}
D = {\rm max} \lbrace \partial(x,y) \vert x,y \in X \rbrace.
\end{align*}
 We have $D\geq 1$ since $\vert X \vert \geq 2$.
Define a matrix $A \in {\rm Mat}_X(\mathbb R)$ that has $(x,y)$-entry
\begin{align*}
A_{x,y} = \begin{cases} 1 & \hbox{\rm if $\partial(x,y)=1$}; \\
                                      0 & \hbox{\rm if $\partial(x,y)\not=1$}
                 \end{cases} \qquad \qquad x,y \in X.
\end{align*}
We call $A$ the {\it adjacency matrix} of $\Gamma$. The matrix $A$ is real and symmetric, so $A$ is diagonalizable.
By an {\it eigenvalue of $\Gamma$} we mean a root of the minimal polynomial of $A$. Let $\theta$ denote an eigenvalue of $\Gamma$.
By the {\it multiplicity of $\theta$} we mean the dimension of the $\theta$-eigenspace of $A$. By the {\it primitive idempotent of $\Gamma$ associated with $\theta$},
we mean the matrix $E \in {\rm Mat}_X(\mathbb R)$ that acts as the identity on the $\theta$-eigenspace of $A$, and as $0$ on every other eigenspace of $A$. 
Note that $E^2=E$ and $AE=\theta E = E A$. The $\theta$-eigenspace of $A$ is given by
\begin{align*}
EV = {\rm Span}\lbrace E \hat x \vert x \in X\rbrace.
\end{align*}
For $x \in X$ and $0 \leq i \leq D$, define the set
\begin{align*}
\Gamma_i(x) = \lbrace y \in X\vert \partial(x,y)=i\rbrace.
\end{align*}
We abbreviate $\Gamma(x)=\Gamma_1(x)$. For an integer $\kappa \geq 1$, we say that $\Gamma$ is {\it regular with valency $\kappa$} whenever 
$\vert \Gamma(x) \vert=\kappa$ for all $x \in X$.
The graph $\Gamma$ is called {\it distance-regular} whenever for all $0 \leq h,i,j\leq D$ and $x,y \in X$ at distance $\partial(x,y)=h$, the number
\begin{align*}
p^h_{i,j} = \vert \Gamma_i(x) \cap \Gamma_j(y) \vert
\end{align*}
is independent of $x,y$ and depends only on $h,i,j$.
For the rest of this paper, we assume that $\Gamma$ is distance-regular with diameter $D\geq 3$. The numbers $p^h_{i,j}$ $(0\leq h,i,j\leq D)$ are called
the {\it intersection numbers} of $\Gamma$. We abbreviate
\begin{align*}
c_i = p^i_{1,i-1} \;\;(1 \leq i \leq D), \qquad \quad a_i = p^i_{1,i} \;\;(0 \leq i \leq D), \qquad \quad b_i = p^i_{1,i+1} \;\;(0 \leq i \leq D-1).
\end{align*}
Note that $c_1=1$ and $a_0=0$. The graph $\Gamma$ is regular with valency $\kappa = b_0$. Moreover
\begin{align*}
c_i + a_i + b_i = \kappa \qquad \qquad (0 \leq i \leq D),
\end{align*}
where $c_0=0$ and $b_D=0$. By \cite[p.~128]{bcn} the graph $\Gamma$ has exactly $D+1$ eigenvalues; we denote them by
 \begin{align*}
  \theta_0 > \theta_1 > \cdots > \theta_D.
  \end{align*}
By \cite[Proposition~3.1]{Biggs} we have $\theta_0=\kappa$.
\medskip

\noindent For more information about distance-regular graphs, see \cite{bbit, bannai, bcn, dkt, int}.

 \section{The bilinear forms graph $\Gamma$}
 
 We turn our attention to the following graph $\Gamma=(X, \mathcal R)$. Fix a finite field ${\rm GF}(q)$ with $q\not=2$. Fix integers $D, N$ such that $N>2D\geq 6$.
 The vertex set $X$ consists of the $D \times (N-D)$ matrices that have all entries in ${\rm GF}(q)$. 
 Thus $\vert X \vert = q^{D(N-D)}$.
 Two vertices
 are adjacent whenever their difference has rank one. The graph $\Gamma$ is called the {\it bilinear forms graph }
 and is sometimes denoted $H_q(D, N-D)$, see for example \cite[p.~280]{bcn}.
 The following facts about $\Gamma$ are taken from \cite[Sections~8.4, 9.5]{bcn}.
 The graph $\Gamma$ is distance-regular, with diameter $D$ and intersection numbers
 \begin{align}
 c_i = q^{i-1} \frac{q^i -1}{q-1}, \qquad \quad b_i = \frac{(q^{N-D}-q^i)(q^{D}-q^i)}{q-1} 
 \qquad \qquad (0 \leq i \leq D). \label{eq:cibi}
 \end{align} 
 We have
 \begin{align}
 a_i =      \frac{q^i -1}{q-1} \Bigl           (q^{N-D}  +q^D   -q^i - q^{i-1}-1\Bigr) \qquad \qquad (0 \leq i \leq D). \label{eq:ai}
 \end{align}
 The eigenvalues of $\Gamma$ are
 \begin{align}
 \theta_i = \frac{q^{N-i} + 1 - q^D-q^{N-D}}{q-1} \qquad \qquad (0 \leq i \leq D). \label{eq:theta}
 \end{align}
 \noindent The eigenvalue ordering \eqref{eq:theta}
  is $Q$-polynomial in the sense of \cite[p.~251]{bbit}
 and formally self-dual in the sense of \cite[p.~49]{bcn}.
\medskip

\noindent Let $E$ denote the primitive idempotent of $\Gamma$ associated with $\theta_1$. 
 By \cite[Lemma~3.9 and line (21)]{int}, for $0 \leq i \leq D$ there exists $\theta^*_i \in \mathbb R$ such that
 $\theta^*_i = \vert X \vert E_{x,y}$ for all $x,y \in X$ at distance $\partial(x,y)=i$. We call $\theta^*_i$
 the {\it $i$th dual eigenvalue} of $\Gamma$ with respect to $E$. By \cite[p.~357]{bbit} we have $\theta^*_i = \theta_i$ for $0 \leq i \leq D$.
 Thus 
  \begin{align}
 \theta^*_i =\frac{q^{N-i} + 1 - q^D-q^{N-D}}{q-1}
 \qquad \quad (0 \leq i \leq D). \label{eq:ths}
 \end{align}
 By \cite[Lemma~1.1]{newIneq} we have
 \begin{align} \label{eq:TTR}
  \theta^*_{i-1} c_i + \theta^*_i a_i + \theta^*_{i+1} b_i = \theta_1 \theta^*_i \qquad \qquad (0 \leq i \leq D),
 \end{align}
 where $\theta^*_{-1}$, $\theta^*_{D+1}$ denote indeterminates.
 \medskip
 
 \noindent By \cite[line~(21)]{int} the multiplicity of $\theta_1$ is equal to $\theta^*_0$. Therefore
 \begin{align}
 {\rm dim}\, EV = \theta^*_0 =  \frac{(q^{N-D}-1)(q^D-1)}{q-1}.  
 \label{eq:dimEV}
 \end{align}
 
 \begin{lemma} \label{lem:xyip} {\rm (See \cite[Lemma~1.1]{newIneq}.)} 
 For $x,y \in X$ we have
 \begin{align*}
 \langle E\hat x, E \hat y \rangle = \vert X \vert^{-1} \theta^*_i,
 \end{align*}
 where $i = \partial(x,y)$.
 \end{lemma}

 \begin{lemma} \label{lem:xyind}
  {\rm (See \cite[Section~4]{nortonBalanced}.)}
 For distinct $x,y \in X$ the vectors $E\hat x$, $E\hat y$ are linearly independent.
 \end{lemma}
 
  \noindent We finish this section with a few comments. 
  \medskip
  
\noindent  By \cite[p.~281]{bcn}, the graph $\Gamma$ is distance-transitive in the sense of \cite[p.~136]{bcn}.
\medskip

\noindent  In the literature on distance-regular graphs, there are some parameters $q^h_{i,j}$ $(0 \leq h,i,j\leq D)$ called the Krein parameters;
see for example \cite[Section~2.3]{bcn}.
By \cite[p.~49]{bcn} and formal self-duality, for $\Gamma$ we have $q^h_{i,j} = p^h_{i,j}$ for $0 \leq h,i,j\leq D$. In particular
\begin{align}
q^1_{1,1} = a_1 = q^{N-D} +q^D    -q-2. 
 \label{eq:q111}
\end{align}
 
 \noindent For more information about the bilinear forms graph, see  \cite{delsarte, qtet,kim1,tanaka, alt}.

\section{The local graph with respect to a vertex  in $\Gamma$}
We continue to discuss the bilinear forms graph $\Gamma=(X, \mathcal R)$.
Throughout this section, we fix a vertex $x \in X$. Recall that $\Gamma(x)$ is the set of
vertices in $X$ that are adjacent to $x$.
Note that
\begin{align}
\vert \Gamma(x) \vert = \kappa =  \frac{(q^{N-D}-1)(q^D-1)}{q-1}.       \label{eq:valency}
\end{align}
 The vertex subgraph of $\Gamma$ induced on $\Gamma(x)$
is called the {\it local graph of $\Gamma$ with respect to $x$}. For notational convenience, we denote this local graph by $\Gamma(x)$.
\medskip

\noindent By construction, the local graph $\Gamma(x)$ is regular with valency $a_1$.
To describe this graph in more detail, we use the notation
 \begin{align*}
 \lbrack n \rbrack = \frac{q^n-1}{q-1} \qquad \qquad n \in \mathbb N.
 \end{align*}
 By \cite[Section~6.1]{hobart}  the local graph $\Gamma(x)$ is a $(q-1)$-clique extension of a Cartesian product of complete graphs
 $K_{\lbrack D \rbrack} \times K_{\lbrack N-D\rbrack}$. 
  In the next result, we give the spectrum of the local graph $\Gamma(x)$.

\begin{lemma} \label{lem:localSpec} {\rm (See \cite[Theorem~6.1]{hobart}.)} 
For the local graph $\Gamma(x)$ the eigenvalues and
their multiplicities are given in the table below:
\begin{align*} 
\begin{tabular}[t]{c|c}
{\rm eigenvalue }& {\rm multiplicity}
 \\
 \hline
$a_1 = q^{N-D}+q^D-q-2$ & $1$ \\
$q^{N-D}-q-1$  &$ \frac{q^D-q}{q-1}$ \\
$q^D-q-1$ & $\frac{q^{N-D}-q}{q-1}$ \\
$-1$ & $ \frac{(q^D-1)(q^{N-D}-1)(q-2)}{(q-1)^2}$ \\
$-q$ & $\frac{(q^D-q)(q^{N-D}-q)}{(q-1)^2}$
    \end{tabular}
\end{align*}
\end{lemma}

\noindent We have some comments about the local graph $\Gamma(x)$.

 \begin{lemma} \label{lem:comA} 
 The vectors 
 \begin{align} \label{eq:localBasis}
  E \hat y \qquad \qquad  \bigl( y \in \Gamma(x) \bigr)
 \end{align} 
   form a basis for $EV$.
 \end{lemma}
 \begin{proof} 
 By  \eqref{eq:dimEV} and  \eqref{eq:valency}, the number of vectors listed in \eqref{eq:localBasis} is equal to the dimension of $EV$.
So it suffices to show that the vectors listed in \eqref{eq:localBasis} are linearly independent.
Let $\tilde A$ denote the adjacency matrix for the local graph $\Gamma(x)$. Let $\tilde I$ and $\tilde J$ denote square matrices of the same size as $\tilde A$, with
$\tilde I$ an identity matrix and $\tilde J$ having all entries 1.
By Lemma \ref{lem:xyip}, for the vectors \eqref{eq:localBasis} the inner product matrix is equal to $\vert X \vert^{-1} $ times
\begin{align}
\theta^*_0 \tilde I + \theta^*_1 \tilde A + \theta^*_2 (\tilde J - \tilde A - \tilde I).
\label{eq:LGM}
\end{align}
It suffices to show that the matrix \eqref{eq:LGM} is nonsingular. To do this, we show that
each eigenvalue of  \eqref{eq:LGM} is nonzero.
The eigenvalues of $\tilde A$ are given in Lemma \ref{lem:localSpec}.
There exists a unique polynomial $f$ of degree 4 that has roots
\begin{align}
q^{N-D}-q-1, \qquad \quad q^D-q-1, \qquad \quad -1, \qquad \quad -q \label{eq:nontriv}
\end{align}
and $f(a_1)=\kappa$. We have  $f(\tilde A)=\tilde J$ by
 Lemma \ref{lem:localSpec} and
\cite[Corollary~3.3]{Biggs}.
The matrix  \eqref{eq:LGM} is a polynomial in $\tilde A$; the eigenvalues of 
  \eqref{eq:LGM} 
  are obtained by applying this polynomial to the eigenvalues of $\tilde A$.
For the eigenvalue $a_1$ of $\tilde A$, the corresponding eigenvalue of  \eqref{eq:LGM} is equal to
\begin{align*}
&\theta^*_0 + \theta^*_1 a_1 + \theta^*_2 \bigl(f(a_1)-a_1-1\bigr) 
= \theta^*_0 + \theta^*_1 a_1 + \theta^*_2 \bigl(\kappa-a_1-1\bigr) \\
&= \theta^*_0 + \theta^*_1 a_1 + \theta^*_2 b_1 
= \theta_1 \theta^*_1 
\not=0.
\end{align*}
\noindent For an eigenvalue $\eta$ of $\tilde A$ listed in \eqref{eq:nontriv},  the corresponding eigenvalue of  \eqref{eq:LGM} is equal to
\begin{align*}
& \theta^*_0 + \theta^*_1 \eta + \theta^*_2 \bigl(f(\eta)-\eta-1\bigr) 
 =  \theta^*_0 + \theta^*_1 \eta + \theta^*_2 \bigl(0-\eta-1\bigr) \\
& = (\theta^*_1-\theta^*_2)\biggl( \eta + \frac{\theta^*_0-\theta^*_2}{\theta^*_1 - \theta^*_2} \biggr) 
 = (\theta^*_1-\theta^*_2)( \eta + q+1) 
\not=0.
 \end{align*}
We have shown that the matrix \eqref{eq:LGM} is nonsingular. The result follows.
 \end{proof}

 \begin{lemma} \label{eq:Gsum} For $x \in X$,
 \begin{align*}
  \theta_1 E\hat x = \sum_{y \in \Gamma(x)} E \hat y.
  \end{align*}
 \end{lemma}
 \begin{proof} Compare column $x$ on each side of $\theta_1 E = EA$.
 \end{proof}

 \section{The $y$-partition of $\Gamma(x)$}
 
 We continue to discuss the bilinear forms graph $\Gamma=(X, \mathcal R)$.
 Throughout this section we fix $x,y \in X$ and write $k=\partial(x,y)$. To avoid degenerate situations, we always assume $2 \leq k \leq D-1$.
 In this section, we use $y$ to obtain a partition 
  $\lbrace O_i\rbrace_{i=1}^6$ of $\Gamma(x)$. We call this partition the $y$-partition of $\Gamma(x)$.
 We show that the $y$-partition of $\Gamma(x)$ is equitable in the sense of \cite[p.~436]{bcn}.
  
\begin{proposition} \label{prop:parts}
There exists a partition of $\Gamma(x)$ into $6$ nonempty subsets 
\begin{align*}
O_i \qquad \qquad (1 \leq i \leq 6)
\end{align*}
such that the following {\rm (i)--\rm (vi)}  holds.
\begin{enumerate}
\item[\rm (i)] 
$O_1= \Gamma(x) \cap \Gamma_{k-1}(y)$.
\item[\rm (ii)] The set $O_2$ consists of the vertices $z \in \Gamma(x)\cap \Gamma_k(y)$ such that 
\begin{align*}
\vert \Gamma(x) \cap \Gamma(z) \cap \Gamma_{k-1}(y) \vert = 2 q^{k-1}, \qquad \qquad 
\vert \Gamma(x) \cap \Gamma(z) \cap \Gamma_{k+1}(y) \vert = 0.
 \end{align*}
\item[\rm (iii)] The set $O_3$ consists of the vertices $z \in  \Gamma(x)\cap \Gamma_k(y)$ such that 
\begin{align*}
\vert \Gamma(x) \cap \Gamma(z) \cap \Gamma_{k-1}(y) \vert = 2 q^{k-1}-1, \qquad \qquad 
\vert \Gamma(x) \cap \Gamma(z) \cap \Gamma_{k+1}(y) \vert = 0.
\end{align*}
\item[\rm (iv)] The set  $O_4$ consists of the vertices $z \in  \Gamma(x) \cap \Gamma_k(y)$ such that 
\begin{align*}
\vert \Gamma(x) \cap \Gamma(z) \cap \Gamma_{k-1}(y) \vert = q^{k-1}, \qquad \qquad 
\vert \Gamma(x) \cap \Gamma(z) \cap \Gamma_{k+1}(y) \vert = q^D-q^k.
\end{align*}
\item[\rm (v)] The set $O_5$ consists of the vertices $z \in  \Gamma(x) \cap \Gamma_k(y)$ such that 
\begin{align*}
\vert \Gamma(x) \cap \Gamma(z) \cap \Gamma_{k-1}(y) \vert = q^{k-1}, \qquad \qquad 
\vert \Gamma(x) \cap \Gamma(z) \cap \Gamma_{k+1}(y) \vert = q^{N-D}-q^k.
\end{align*}
\item[\rm (vi)]
$O_6= \Gamma(x) \cap \Gamma_{k+1}(y)$.
\end{enumerate}
\end{proposition}
\begin{proof} By \cite[Proposition~6.1]{williford} and since $\Gamma$ is distance-transitive.
\end{proof}

\noindent The following definition is motivated by \cite[Section~6]{seong}.
\begin{definition}\label{def:xyPartition} \rm
The partition of $\Gamma(x)$ from Proposition \ref{prop:parts}  is called the {\it $y$-partition of $\Gamma(x)$}.
\end{definition}

 \noindent The $y$-partition of $\Gamma(x)$ is roughly described as follows.
 \begin{lemma} \label{lem:rough}
 The following {\rm (i)--(iii)} hold:
 \begin{enumerate}
 \item[\rm (i)] $O_1 = \Gamma(x) \cap \Gamma_{k-1}(y)$;
 \item[\rm (ii)] $O_2, O_3, O_4, O_5$ form a partition of $\Gamma(x) \cap \Gamma_k (y)$;
  \item[\rm (iii)] $O_6 = \Gamma(x) \cap \Gamma_{k+1} (y)$.
 \end{enumerate}
 \end{lemma}
 \begin{proof} By Proposition  \ref{prop:parts}.
 \end{proof}
 
 \noindent See \cite[Sections~6, 7]{williford} for some data involving the $y$-partition of $\Gamma(x)$.
 \medskip
 
 \noindent Our next general goal is to show that the $y$-partition of $\Gamma(x)$ is equitable.
 
 \begin{lemma} \label{lem:Osize}
We have
 \begin{align*} 
& \vert O_1 \vert = \frac{q^{k-1} (q^k - 1) }{q - 1},  \qquad \qquad \quad
                       \vert O_2 \vert = \frac{(q^k - 1)(q^{k - 1} - 1)}{q - 1}, \\
                &     \vert O_3\vert= \frac{q^{k-1}(q^k - 1) (q - 2)}{q - 1},  \qquad \quad             
\vert O_4 \vert= \frac{(q^{N-D} - q^k)(q^k - 1)}{q - 1}, \\           
&\vert O_5 \vert= \frac{(q^D - q^k)(q^k - 1)}{q - 1}, \qquad \qquad
                       \vert O_6 \vert =\frac{(q^{N-D} - q^k)(q^D - q^k)}{q - 1}.        
\end{align*}
\end{lemma}
\begin{proof} By combinatorial counting using the data in \cite[Sections~6,~7]{williford}.
\end{proof}

 
 \begin{proposition} \label{lem:Centry}
 The partition $\lbrace O_i \rbrace_{i=1}^6$ is equitable. 
  Define a $6 \times 6$ matrix $C$ as follows: for $1\leq i,j\leq 6$ the entry $C_{i,j}$ is the number of vertices in $O_j$ that are adjacent to 
 a given vertex in $O_i$. We have $C=$
\begin{tiny}
 \begin{align*}
  \begin{pmatrix}
  2(q^{k-1} - 1) &2(q^{k - 1} - 1)& (2q^{k - 1} - 1)(q - 2)& q^{N-D} - q^k& q^D - q^k& 0 \\
  2q^{k - 1}&  2 q^{k-1} - 2 - q& 2 q^{k-1}(q - 2)& q^{N-D} - q^k& q^D- q^k& 0  \\
  2q^{k-1} - 1& 2(q^{k-1} - 1)& 2q^k - 4q^{k-1}- q + 1& q^{N-D} - q^k& q^D - q^k& 0 \\
  q^{k-1} & q^{k-1}- 1& q^{k-1}(q - 2) & q^{N-D} - q - 1& 0& q^D  - q^k \\
  q^{k-1} & q^{k-1}  - 1& q^{k-1} (q - 2)& 0& q^D - q - 1& q^{N-D} - q^k \\
  0& 0& 0& q^k - 1& q^k - 1& q^{N-D} + q^D - 2 q^k - q
  \end{pmatrix}.
 \end{align*}
\end{tiny}
 \end{proposition}
 \begin{proof}
 By combinatorial counting  using the data in \cite[Sections~6,~7]{williford}.
 \end{proof}

  \begin{lemma} \label{lem:Ceig}  The following matrix is symmetric:
  \begin{align*}
  {\rm diag}\bigl( \vert O_1 \vert, \vert O_2 \vert, \vert O_3 \vert, \vert O_4 \vert, \vert O_5 \vert, \vert O_6 \vert \bigr)  C.
  \end{align*}
  \end{lemma}
  \begin{proof} For $1 \leq i,j\leq 6$ we count the number of ordered pairs $(z,z')$ of vertices such that $z \in O_i$ and $z' \in O_j$ and $z,z'$ are adjacent.
  Performing this count in two ways, we obtain
  \begin{align*}
  \vert O_i \vert C_{i,j} = \vert O_j \vert C_{j,i}.
  \end{align*}
  The result follows.
  \end{proof}
  
  \begin{lemma} \label{lem:CSpec} 
  The matrix $C$ is diagonalizable.  Its eigenvalues and
their multiplicities are given in the table below:
\begin{align*} 
\begin{tabular}[t]{c|c}
{\rm eigenvalue }& {\rm multiplicity}
 \\
 \hline
$a_1=q^{N-D} +q^D-q-2$ & $1$ \\
$q^{N-D}-q-1$  &$1$ \\
$q^D-q-1$ & $1$ \\
$-1$ & $1$ \\
$-q$ & $2$
    \end{tabular}
\end{align*}
\end{lemma}
\begin{proof} By linear algebra.
\end{proof}

\begin{definition} \label{def:vth}
\rm For notational convenience, we define
\begin{align*}
&\vartheta_1 = a_1=q^{N-D}+q^D-q-2, \qquad \qquad \vartheta_2 = q^{N-D}-q-1, \\
&\vartheta_3 = q^D-q-1, \qquad \qquad \vartheta_4 = -1, \qquad \qquad \vartheta_5=\vartheta_6=-q.
\end{align*}
\end{definition}

\noindent Our next general goal is to display some column eigenvectors for $C$.

\begin{definition} \label{def:Hmatrix} \rm Let $H$ denote the following $6 \times 6$ matrix:
\begin{align*}
\begin{pmatrix}
1 &\frac{q^D-q^k}{q-1}&\frac{q^{N-D}-q^k}{q-1}&q-2&\frac{q^N-q^{N-D+1}-q^{D+1} + q^{k+1} -q^k+q}{(q-1)^2}&\frac{(q^D-q^k)(q^{N-D}-q^k)(q^{k-1}-1)}{q^k(q-1)} \\
1 &\frac{q^D-q^k}{q-1}&\frac{q^{N-D}-q^k}{q-1}&0&\frac{q^k-q^N}{q-1}& \frac{(q^D-q^k)(q^{N-D}-q^k)}{q(q-1)} \\
1 &\frac{q^D-q^k}{q-1}&\frac{q^{N-D}-q^k}{q-1}&-1&\frac{q^N-q^{N-D+1}-q^{D+1} + q^{k+1} -q^k+q}{(q-1)^2}&\frac{(q^D-q^k)(q^{N-D}-q^k)(q^{k-1}-1)}{q^k(q-1)} \\
1 &\frac{q^D-q^k}{q-1}& - \frac{q^k-1}{q-1}&0&\frac{q^k-q^D}{q-1}&- \frac{(q^D-q^k)(q^{k-1}-1)}{q-1}\\
1 &- \frac{q^k-1}{q-1}&\frac{q^{N-D}-q^k}{q-1}&0&\frac{q^k-q^{N-D}}{q-1}& - \frac{(q^{N-D}-q^k)(q^{k-1}-1)}{q-1}\\
1 &- \frac{q^k-1}{q-1}&- \frac{q^k-1}{q-1}&0&\frac{q^k-1}{q-1}& \frac{(q^k-1)(q^{k-1}-1)}{q-1}
\end{pmatrix}
\end{align*}
\end{definition}

\begin{lemma} \label{lem:Hcol} The columns of $H$ are eigenvectors for $C$. More precisely,
\begin{align*}
 C H = H {\rm diag} \bigl(\vartheta_1, \vartheta_2, \vartheta_3, \vartheta_4, \vartheta_5, \vartheta_6\bigr). 
\end{align*}
\end{lemma}
  \begin{proof} The above equation is checked by matrix multiplication.
  \end{proof}

  \begin{lemma} \label{lem:HGorthog} We have
  \begin{align*}
 & H^t   {\rm diag}\bigl( \vert O_1 \vert, \vert O_2 \vert, \vert O_3 \vert, \vert O_4 \vert, \vert O_5 \vert, \vert O_6 \vert \bigr) H
 = {\rm diag} \bigl( \eta_1, \eta_2, \eta_3, \eta_4, \eta_5, \eta_6\bigr),
   \end{align*}
   where 
   \begin{align*}
   \eta_1 &= \frac{(q^D-1)(q^{N-D}-1)}{q-1}, \\
   \eta_2 &= \frac{(q^D-q^k)(q^k-1)(q^D-1)(q^{N-D}-1)}{(q-1)^3}, \\
     \eta_3 &= \frac{(q^{N-D}-q^k)(q^k-1)(q^D-1)(q^{N-D}-1)}{(q-1)^3}, \\
     \eta_4 & = (q-2)q^{k-1} (q^k-1), \\
     \eta_5 &= \frac{ (q^{N-D}-q)(q^{D-1}-1)+q^{N-k}(q^{k-1}-1)(q-1)}{(q-1)^2}\\
     & \qquad \qquad \times \frac{ q^k(q^k-1)(q^D-1)(q^{N-D}-1)}{(q-1)^2},\\
      \eta_6 &= \frac{ (q^{N-D}-q)(q^{D-1}-1)+q^{N-k}(q^{k-1}-1)(q-1)}{(q-1)^2} \\
      & \qquad \qquad \times \frac{ (q^k-1)(q^{k-1}-1)(q^D-q^k)(q^{N-D}-q^k)}{(q-1)q}.
   \end{align*}
  \end{lemma}
\begin{proof} By matrix multiplication.
\end{proof}
\begin{lemma} \label{lem:Hinv} The matrix $H$ is invertible.
\end{lemma}
\begin{proof} By Lemma \ref{lem:HGorthog} and since $\eta_i \not=0$ for $1 \leq i \leq 6$.
\end{proof}

  \section{The subspace $S$}
   We continue to discuss the bilinear forms graph $\Gamma=(X, \mathcal R)$ and the fixed vertices $x,y \in X$ at distance $\partial(x,y)=k$. Recall
   the $y$-partition $\lbrace O_i \rbrace_{i=1}^6$ of $\Gamma(x)$ from Definition \ref{def:xyPartition}.
   In this section, we use the partition $\lbrace O_i \rbrace_{i=1}^6$ to obtain a subspace $S$ of $EV$. We describe $S$ in various ways.
 \medskip
  
  \noindent For a subset $ \Omega \subseteq X$ define $\widehat \Omega = \sum_{z \in \Omega} \hat z$.
  
  \begin{lemma} \label{lem:6b} The vectors $\lbrace E \widehat O_i\rbrace_{i=1}^6$ are linearly independent.
  \end{lemma}
  \begin{proof} By Lemma \ref{lem:comA} and since the sets $\lbrace  O_i\rbrace_{i=1}^6$ are nonempty and partition $\Gamma(x)$.
  \end{proof} 
  
  \begin{definition}\label{def:N} \rm Let $S=S(x,y)$ denote the subspace of $EV$ with basis  $\lbrace E \widehat O_i\rbrace_{i=1}^6$.
  The dimension of $S$ is $6$.
  \end{definition}
  
  \noindent Our next general goal is to show that $S$ contains $E\hat x$ and $E \hat y$.
  
  \begin{lemma} \label{lem:Osum} We have
  \begin{align*}
 \theta_1  E \hat x = \sum_{i=1}^6 E \widehat O_i.
 \end{align*}
  \end{lemma}
  \begin{proof} By Lemma \ref{eq:Gsum} and since $\lbrace O_i \rbrace_{i=1}^6$ partition $\Gamma(x)$.
  \end{proof}
  
 \begin{lemma} \label{lem:dep} We have
 \begin{align}
 0 =  E{\hat y} +\frac{1-q^{1-k}}{q-1} E{\hat x} - q^{1-k} E \widehat O_1 + q^{1-k} E \widehat O_2. \label{eq:dep}
 \end{align}
 \end{lemma}
 \begin{proof} Let $\xi$ denote the vector on the right-hand side of  \eqref{eq:dep}. We will show $\xi=0$ by showing
 $\Vert \xi \Vert^2 = 0$. In the following calculations we use \eqref{eq:ths} and Lemmas \ref{lem:xyip},
 \ref{lem:rough}, \ref{lem:Osize},  \ref{lem:Centry}.
  First we show that $\langle E{\hat x}, \xi \rangle =0$.
 We have
 \begin{align*}
 \vert X \vert \langle E{\hat x}, \xi \rangle = \theta^*_k 
+\frac{1-q^{1-k}}{q-1} \theta^*_0  - q^{1-k}  \vert O_1\vert  \theta^*_1 + q^{1-k} \vert O_2 \vert \theta^*_1 = 0.
 \end{align*}
 Next we show that $\langle E{\hat y}, \xi \rangle =0$.
 We have
 \begin{align*}
 \vert X \vert \langle E{\hat y}, \xi \rangle = \theta^*_0 
+\frac{1-q^{1-k}}{q-1} \theta^*_k  - q^{1-k}  \vert O_1\vert  \theta^*_{k-1} + q^{1-k} \vert O_2 \vert \theta^*_k = 0.
 \end{align*}
 Next we show that $\langle E{\hat z}, \xi \rangle =0$ for $z \in O_1$. We have
  \begin{align*}
 \vert X \vert \langle E{\hat z}, \xi \rangle &= \theta^*_{k-1}
+\frac{1-q^{1-k}}{q-1} \theta^*_1  - q^{1-k} \Bigl( \theta^*_0 + C_{1,1} \theta^*_1+  \bigl(\vert O_1\vert -C_{1,1}-1\bigr) \theta^*_2 \Bigr) 
\\
& \qquad + q^{1-k} \Bigl( C_{1,2} \theta^*_1 + \bigl( \vert O_2\vert -C_{1,2} \bigr) \theta^*_2 \Bigr) = 0.
 \end{align*}
 Next we show that $\langle E{\hat w}, \xi \rangle =0$ for $w \in O_2$. We have
  \begin{align*}
 \vert X \vert \langle E{\hat w}, \xi \rangle &= \theta^*_{k}
+\frac{1-q^{1-k}}{q-1} \theta^*_1  - q^{1-k} \Bigl(  C_{2,1} \theta^*_1+  \bigl(\vert O_1\vert -C_{2,1}\bigr) \theta^*_2 \Bigr) 
\\
& \qquad + q^{1-k} \Bigl( \theta^*_0 + C_{2,2} \theta^*_1 + \bigl( \vert O_2\vert -C_{2,2} -1\bigr) \theta^*_2 \Bigr) = 0.
 \end{align*}
By these comments $\Vert \xi \Vert^2=0$, so $\xi=0$.
 \end{proof}
 
 \begin{corollary} \label{cor:xy} The subspace $S$ contains $E\hat x$ and $E\hat y$.
 \end{corollary}
 \begin{proof} By Lemmas \ref{lem:Osum}, \ref{lem:dep} and since $\theta_1 \not=0$.
 \end{proof}
 
 \section{Strengthening the balanced set condition}
 We continue to discuss the bilinear forms graph $\Gamma=(X, \mathcal R)$ and the fixed vertices $x,y \in X$ at distance $\partial(x,y)=k$.
 Recall the $y$-partition $\lbrace O_i \rbrace_{i=1}^6$ of $\Gamma(x)$  from Definition  \ref{def:xyPartition}.  Let $\lbrace O'_i \rbrace_{i=1}^6$ denote the $x$-partition of $\Gamma(y)$.
 By the balanced set condition \cite[Theorem~1.1]{QPchar}, for $i\in \lbrace 1,6 \rbrace$ the vector $E \widehat O_i- E\widehat O'_i$ is a scalar multiple of $E\hat x - E \hat y$.
 In this section, we show that for $1 \leq i \leq 6$ the vector $E \widehat O_i- E\widehat O'_i$ is a scalar multiple of $E\hat x - E \hat y$. 
 \medskip

 \begin{theorem}
 \label{thm:bbalanced} 
For $1 \leq i \leq 6$ we have
 \begin{align*}
 E \widehat O_i- E \widehat O'_i = \lambda_i ( E \hat x - E \hat y),
 \end{align*}
 where
 \begin{align}
 &\lambda_i = \vert O_i \vert \frac{\theta^*_1 - \theta^*_k}{\theta^*_0 - \theta^*_k} \qquad \qquad (2 \leq i \leq 5), \label{eq:lam2345}\\
   &\lambda_1 = \vert O_1 \vert \frac{\theta^*_1 - \theta^*_{k-1}}{\theta^*_0 - \theta^*_k}, \qquad \qquad
 \lambda_6 = \vert O_6 \vert \frac{\theta^*_1 - \theta^*_{k+1}}{\theta^*_0 - \theta^*_k}. \label{eq:lam16}
 \end{align}
 \end{theorem}
 
 \noindent We will prove Theorem  \ref{thm:bbalanced} shortly.
 \medskip
 
 \noindent For notational convenience, for $1 \leq i \leq 6$ we define a scalar $\varepsilon(i)$ as follows:
 \begin{align}
 \begin{tabular}[t]{c|cccccc}
 $i$ &$1$ &$2$&$3$&$4$&$5$&$6$\\
 \hline
$\varepsilon(i)$ &$-1$&$0$&$0$&$0$&$0$&$1$
    \end{tabular}
    \label{eq:eps}
 \end{align}
 Then \eqref{eq:lam2345}, \eqref{eq:lam16} become
  \begin{align}
 \lambda_i =  \vert O_i \vert \frac{\theta^*_1 - \theta^*_{k+\varepsilon(i)}}{\theta^*_0 - \theta^*_k}
 \qquad \qquad (1 \leq i \leq 6). \label{eq:lambdaNotation}
 \end{align}

 \begin{lemma} \label{lem:step1}
 For $1 \leq i \leq 6$ the vector
  \begin{align*}
 E \widehat O_i- E \widehat O'_i - \lambda_i ( E \hat x - E \hat y)       
 \end{align*}
 is orthogonal to $E \hat x$ and $E\hat y$.
 \end{lemma}
 \begin{proof} 
 By Lemma  \ref{lem:xyip} and \eqref{eq:lambdaNotation},
 \begin{align*}
\vert X \vert  \Bigl \langle E\hat x,  E \widehat O_i- E \widehat O'_i - \lambda_i ( E \hat x - E \hat y) \Bigr\rangle 
&= \vert O_i \vert \theta^*_1 - \vert O'_i \vert \theta^*_{k+\varepsilon(i)} - \lambda_i (\theta^*_0 - \theta^*_k) \\
&= \vert O_i \vert \theta^*_1 - \vert O_i \vert \theta^*_{k+\varepsilon(i)} - \lambda_i (\theta^*_0 - \theta^*_k) \\
&= 0.
 \end{align*}
 Similarly,
  \begin{align*}
\vert X \vert  \Bigl \langle E\hat y,  E \widehat O_i- E \widehat O'_i - \lambda_i ( E \hat x - E \hat y) \Bigr \rangle 
&= \vert O_i \vert \theta^*_{k+\varepsilon(i)} - \vert O'_i \vert \theta^*_1 - \lambda_i (\theta^*_k - \theta^*_0) \\
&= \vert O_i \vert \theta^*_{k+\varepsilon(i)} - \vert O_i \vert \theta^*_1 - \lambda_i (\theta^*_k - \theta^*_0) \\
&= 0.
 \end{align*}
 The result follows.
 \end{proof}
 
  \noindent We have a comment.
 \begin{lemma}
 \label{lem:lambdaValues}
 The scalars $\lbrace \lambda_i \rbrace_{i=1}^6$ from Theorem \ref{thm:bbalanced} are given by
 \begin{align*}
& \lambda_1 = q^k \frac{q^{k-2}-1}{q-1}, \qquad \qquad 
 \lambda_2 = \frac{(q^{k-1}-1)^2}{q-1}, \\
& \lambda_3 = q^{k-1} \frac{(q^{k-1}-1)(q-2)}{q-1}, \qquad \qquad 
 \lambda_4 = \frac{(q^{N-D}-q^k)(q^{k-1}-1)}{q-1}, \\
& \lambda_5 = \frac{(q^D-q^k)(q^{k-1}-1)}{q-1}, \qquad \qquad 
 \lambda_6= \frac{(q^D-q^k)(q^{N-D}-q^k)}{q(q-1)}.
 \end{align*}
 \end{lemma}
 \begin{proof} Evaluate \eqref{eq:lam2345}, \eqref{eq:lam16} using \eqref{eq:ths}
  and Lemma \ref{lem:Osize}.
 \end{proof}

 \noindent We now describe how the sets $\lbrace O_i \rbrace_{i=1}^6$ and  $\lbrace O'_i \rbrace_{i=1}^6$ are related.
 This description will involve some $6 \times 6$ matrices 
 \begin{align*}
 D^{(k-2)}, \quad D^{(k-1)}, \quad D^{(k)}, \quad  D^{(k+1)}, \quad D^{(k+2)}.
 \end{align*}

 \begin{proposition}\label{def:Dmatrix} 
  For $k-2 \leq \ell \leq k+2$ 
 the matrix $D^{(\ell)}$  given below has the following feature.
  For $1\leq i,j\leq 6$ the $(i,j)$-entry $D^{(\ell)}_{i,j}$ is the number of vertices in $O'_j$ that are at distance $\ell$ from a given vertex in $O_i$.
%
We have
\begin{align*}
D^{(k-2)} = \begin{pmatrix}
\frac{q^{k-2}(q^{k-1}-1)}{q-1}&0&0&0&0&0 \\
 0&0&0&0&0&0 \\
 0&0&0&0&0&0 \\
 0&0&0&0&0&0 \\
 0&0&0&0&0&0 \\
 0&0&0&0&0&0 \\
\end{pmatrix},
\end{align*}

\begin{landscape}
    
\maketitle

\vspace{1in}
\resizebox{9in}{!}{

$D^{(k-1)} = \begin{pmatrix}
2q^{k-2}(q^{k-1}-1) &(q^{k-1}-1)(q^{k-2} + \frac{q^{k-1}-1}{q-1})&\frac{q^{k-1}-1}{q-1}q^{k-2}(2q-1)(q-2)&\frac{(q^{k-1}-1)(q^{N-D}-q^k)}{q-1}& \frac{(q^{k-1}-1)(q^D-q^k)}{q-1}&0\\
 q^{k-1}(q^{k-2} + \frac{q^{k-1}-1}{q-1})&q^{2k-3}&q^{2k-3}(q-2)&0&0&0 \\
 \frac{q^{k-1}-1}{q-1}q^{k-2}(2q-1)&q^{k-2}(q^{k-1}-1)&q^{k-2}(q^k-2q^{k-1}+2)&0&0&0 \\
\frac{(q^{k-1}-1)q^{k-1}}{q-1} &0&0&q^{2(k-1)}&0&0 \\
\frac{(q^{k-1}-1)q^{k-1}}{q-1} &0&0&0&q^{2(k-1)}&0 \\
 0&0&0&0&0&0 \\
\end{pmatrix},$
}

\vspace{1in}
\resizebox{9in}{!}{
$D^{(k)} = \begin{pmatrix}
q^{k-2}(q^k-q^{k-1}+1) &(q^{k-1}-1)(q^{k-1}-q^{k-2})&q^{k-2}(q^k-q^{k-1}+1)(q-2)&q^{k-1}(q^{N-D}-q^k)&q^{k-1}(q^D-q^k)& \frac{(q^{N-D}-q^k)(q^D-q^k)}{q-1}\\
 q^{2k-3}(q-1)&\frac{(q^k-1)(q^{k-1}-1)}{q-1}-q^{2k-3}&(q-2)q^{k-1}\frac{q^k-q^{k-1}+q^{k-2}-1}{q-1}&\frac{(q^{N-D}-q^k)(q^k-1)}{q-1}&\frac{(q^D-q^k)(q^k-1)}{q-1}&0 \\
 q^{k-2}(q^k-q^{k-1}+1)&(q^{k-1}-1)\frac{q^k-q^{k-1}+q^{k-2}-1}{q-1}&q^{k-2}(q-2)(q^k+\frac{q^{k-1}-1}{q-1})-q^{k-1}&\frac{(q^{N-D}-q^k)(q^k-1)}{q-1}&\frac{(q^D-q^k)(q^k-1)}{q-1}&0 \\
 q^{2(k-1)}&\frac{(q^{k-1}-1)(q^k-1)}{q-1}&\frac{q^{k-1}(q-2)(q^k-1)}{q-1}&\frac{(q^{N-D}-q^k)(q^k-1)}{q-1}-q^{2(k-1)}&\frac{(q^D-q^k)(q^{k-1}-1)}{q-1}&q^{k-1}(q^D-q^k) \\
 q^{2(k-1)}&\frac{(q^{k-1}-1)(q^k-1)}{q-1}&\frac{q^{k-1}(q-2)(q^k-1)}{q-1}&
 \frac{(q^{N-D}-q^k)(q^{k-1}-1)}{q-1}&\frac{(q^D-q^k)(q^k-1)}{q-1}-q^{2(k-1)}&(q^{N-D}-q^k)q^{k-1} \\
 \frac{(q^k-1)q^{k-1}}{q-1}&0&0&q^{k-1}(q^k-1)&q^{k-1}(q^k-1)&(q-1)q^{2k-1}+q^{k-1} \\
\end{pmatrix},$
}

\vspace{1in}
\resizebox{9in}{!}{

$D^{(k+1)} = \begin{pmatrix}
0 &0&0&0&0&0 \\
0&0&0&0&0&\frac{(q^{N-D}-q^k)(q^D-q^k)}{q-1} \\
0 &0&0&0&0&\frac{(q^{N-D}-q^k)(q^D-q^k)}{q-1} \\
0 &0&0&0&(q^D-q^k)q^{k-1}&\frac{(q^D-q^k)(q^{N-D}-2q^k+q^{k-1})}{q-1} \\
0 &0&0&(q^{N-D}-q^k)q^{k-1}&0&\frac{(q^{N-D}-q^k)(q^D-2q^k+q^{k-1})}{q-1} \\
0 &\frac{(q^k-1)(q^{k-1}-1)}{q-1}&\frac{(q^k-1)q^{k-1}(q-2)}{q-1}&\frac{(q^k-1)(q^{N-D}-2q^k+q^{k-1})}{q-1}&\frac{(q^k-1)(q^D-2q^k+q^{k-1})}{q-1}&q^{k-1}(q^{D+1}+q^{N-D+1}-q^{k+2}-2q^{k+1}+q^k-1) \\
 \end{pmatrix},$
}

\end{landscape}

\begin{align*}
D^{(k+2)} = \begin{pmatrix}
 0&0&0&0&0&0 \\
 0&0&0&0&0&0 \\
 0&0&0&0&0&0 \\
 0&0&0&0&0&0 \\
 0&0&0&0&0&0 \\
 0&0&0&0&0&\frac{(q^D-q^{k+1})(q^{N-D}-q^{k+1})}{q-1} \\
 \end{pmatrix}.
 \end{align*}
 \end{proposition}
 \begin{proof} By combinatorial counting using the data in \cite[Sections~6,~7]{williford}.
 \end{proof}
 
 \begin{remark} \rm If $k=D-1$ then $D^{(k+2)}=0$.
 \end{remark}
 
 \noindent Next, we explain how the matrices $D^{(\ell)}$ $(k-2 \leq \ell \leq k+2)$ are related to the matrix $C$ from Proposition \ref{lem:Centry}.
 
 
 \begin{lemma} \label{lem:step2} For $1 \leq i,j\leq 6$ we have
 \begin{align*}
 \delta_{i,j} \theta^*_0 + C_{i,j} \theta^*_1 + \Bigl( \vert O_j \vert - C_{i,j} - \delta_{i,j}\Bigr) \theta^*_2
 - \sum_{\ell=k-2}^{k+2} D^{(\ell)}_{i,j} \theta^*_\ell = \lambda_j \bigl(\theta^*_1 - \theta^*_{k+ \varepsilon(i)}\bigr),
 \end{align*}
 where $\delta_{i,j}$ is the Kronecker delta and $\varepsilon(i)$ is from   \eqref{eq:eps}.
 \end{lemma}
 \begin{proof} Evaluate each side using \eqref{eq:ths} along with Lemmas \ref{lem:Osize}, \ref{lem:Centry},  \ref{lem:lambdaValues} and Proposition \ref{def:Dmatrix}. 
 \end{proof}

 \noindent {\it Proof of Theorem  \ref{thm:bbalanced}.} Define the vector  $\xi =E\widehat O_i - E \widehat O'_i -\lambda_i (E\hat x - E \hat y)$.
 We will show $\xi =0$ by showing  $\Vert \xi \Vert^2 =0$.  By Lemma  \ref{lem:step1} we have $\langle E\hat x, \xi\rangle=0$ and $\langle E \hat y, \xi \rangle =0$.
 Next we show that $\langle E \hat z, \xi\rangle =0$ for $z \in O_i$. Using Lemmas  \ref{lem:xyip},  \ref{lem:step2} we obtain
 \begin{align*}
\vert X \vert  \langle E\hat z, \xi \rangle &=
  \theta^*_0 + C_{i,i} \theta^*_1 + \Bigl( \vert O_i \vert - C_{i,i} - 1\Bigr) \theta^*_2
 - \sum_{\ell=k-2}^{k+2} D^{(\ell)}_{i,i} \theta^*_\ell - \lambda_i \bigl(\theta^*_1 - \theta^*_{k+ \varepsilon(i)}\bigr) \\
 &=0.
 \end{align*}
  Next we show that $\langle E \hat w, \xi\rangle =0$ for $w \in O'_i$. This is obtained by interchanging the roles of $x,y$ in the previous argument.
  By these comments $\Vert \xi \Vert^2 =0$, so $\xi=0$.
$ \hfill \Box$
\medskip

 \begin{proposition} \label{prop:symS} We have $S(x,y) = S(y,x)$.
 \end{proposition} 
 \begin{proof} By Definition \ref{def:N}
  and the construction,
 the vectors $\lbrace E \widehat O'_i \rbrace_{i=1}^6$ form a basis for $S(y,x)$. For $1 \leq i \leq 6$
 we have $E \widehat O'_i \in S(x,y)$ by Corollary \ref{cor:xy} and Theorem  \ref{thm:bbalanced}.
 Therefore $S(y,x) \subseteq S(x,y)$. In this inclusion, equality holds because $S(y,x)$ and $S(x,y)$ have the same dimension.
 \end{proof}
 
 \noindent The following result is for later use.
 \begin{lemma} \label{lem:com} We have
$ \theta_1 = \sum_{i=1}^6 \lambda_i$.
 \end{lemma}
 \begin{proof} This is verified using Lemma  \ref{lem:lambdaValues}. Alternatively, by Lemma \ref{lem:Osum} and the construction
 \begin{align}
 \theta_1 E\hat x = \sum_{i=1}^6 E \widehat O_i, \qquad \qquad \theta_1 E \hat y = \sum_{i=1}^6 E \widehat O'_i.    \label{eq:step1}
 \end{align}
 By \eqref{eq:step1} and Theorem  \ref{thm:bbalanced},
 \begin{align*}
 \theta_1 \bigl(E \hat x - E \hat y\bigr) = \sum_{i=1}^6 \Bigl( E \widehat O_i - E \widehat O'_i \Bigr) = \bigl(E\hat x - E \hat y\bigr) \sum_{i=1}^6 \lambda_i.
 \end{align*}
 The vector $E\hat x - E \hat y$ is nonzero by Lemma  \ref{lem:xyind}. The result follows.
 \end{proof}

 \section{Two orthogonal bases for $S$}
  We continue to discuss the bilinear forms graph $\Gamma=(X, \mathcal R)$ and the fixed vertices $x,y \in X$ at distance $\partial(x,y)=k$.
 Recall the subspace $S=S(x,y) = S(y,x)$ of $EV$ from Definition \ref{def:N}  and Proposition \ref{prop:symS}. In this section,
 we display two orthogonal bases for $S$. We describe how these two bases are related to each other. We also describe how these two bases
 are related to $E{\hat x}$, $E{\hat y}$.
 \medskip
 
 \noindent Recall the matrix $H$ from Definition  \ref{def:Hmatrix}.
 
 \begin{definition} \label{def:hvec} For $1 \leq j \leq 6$ define the vectors
 \begin{align*}
 h_j = \sum_{i=1}^6 H_{i,j} E \widehat O_i, \qquad \qquad h'_j = \sum_{i=1}^6 H_{i,j} E \widehat O'_i.
 \end{align*}
 \end{definition}
 
 \begin{lemma} \label{lem:hbasis} 
 Each of $\lbrace h_j \rbrace_{j=1}^6$, $\lbrace h'_j \rbrace_{j=1}^6$ is a basis for $S$.
 \end{lemma}
 \begin{proof} By Definition \ref{def:N}, Lemma \ref{lem:Hinv},  and Proposition \ref{prop:symS}.
 \end{proof}
  
 \begin{lemma} \label{lem:h1} We have
 \begin{align*}
 h_1 = \theta_1 E{\hat x}, \qquad \qquad h'_1 = \theta_1 E{\hat y}.
 \end{align*}
 \end{lemma}
 \begin{proof} The equation on the left follows from Lemma  \ref{lem:Osum}, Definition \ref{def:hvec}, and $H_{i,1} = 1$ $(1 \leq i \leq 6)$. The
 equation on the right is similarly obtained.
 \end{proof}

 \begin{definition} \label{def:Gmat} Let $G$ denote the $ 6 \times 6$ matrix with entries
 \begin{align*}
  G_{i,j} =\vert X \vert  \langle E \widehat O_i, E\widehat O_j \rangle \qquad \qquad (1 \leq i,j\leq 6).
 \end{align*}
 \end{definition}

 \begin{lemma} \label{lem:Gentries} For $1 \leq i,j\leq 6$ we have
 \begin{align*}
 G_{i,j} &= \vert O_i \vert \Bigl( \delta_{i,j} \theta^*_0 + C_{i,j} \theta^*_1 + \bigl( \vert O_j \vert-C_{i,j} -\delta_{i,j} \bigr) \theta^*_2        \Bigr) \\
    &=\vert O_j \vert \Bigl( \delta_{i,j} \theta^*_0 + C_{j,i} \theta^*_1 + \bigl( \vert O_i \vert-C_{j,i} -\delta_{i,j} \bigr) \theta^*_2        \Bigr).
 \end{align*}
 \end{lemma}
 \begin{proof} Use Lemma \ref{lem:xyip} and Proposition \ref{lem:Centry}.
 \end{proof}

 \begin{lemma} \label{lem:Gp} We have
  \begin{align*}
  G_{i,j} = \vert X \vert \langle E\widehat O'_i, E\widehat O'_j \rangle \qquad \qquad (1 \leq i,j\leq 6).
 \end{align*}
\end{lemma}
 \begin{proof} Interchange the roles of $x,y$ in Definition \ref{def:Gmat} and Lemma \ref{lem:Gentries}.
 \end{proof}
 
 \begin{lemma} \label{lem:HtGH} We have
\begin{align*}
H^t G H = {\rm diag} \bigl( \varepsilon_1 \eta_1, \varepsilon_2 \eta_2,  \varepsilon_3\eta_3, \varepsilon_4\eta_4, \varepsilon_5\eta_5,  \varepsilon_6\eta_6\bigr),
\end{align*}
where  $\lbrace \eta_i \rbrace_{i=1}^6$ are from Lemma \ref{lem:HGorthog} and
\begin{align*}
& \varepsilon_1 = \theta^2_1, \qquad \quad 
\varepsilon_2 = q^{2N-D-2}, \qquad \quad 
\varepsilon_3 = q^{N+D-2}, \\
&
\varepsilon_4 = q^{N-1}, \qquad \quad
\varepsilon_5 =  q^{N-2}, \qquad \quad
\varepsilon_6 =  q^{N-2}.
\end{align*}

\end{lemma}
 \begin{proof}
 By matrix multiplication.
 \end{proof} 
 
\begin{proposition} \label{prop:horthog}
The vectors $\lbrace h_j \rbrace_{j=1}^6$ are mutually orthogonal and
\begin{align*}
\Vert h_j \Vert^2 = \vert X \vert^{-1} \varepsilon_j \eta_j     \qquad \qquad  (1 \leq j \leq 6).
\end{align*}
Moreover, the vectors $\lbrace h'_j \rbrace_{j=1}^6$ are mutually orthogonal and
\begin{align*}
\Vert h'_j \Vert^2 = \vert X \vert^{-1} \varepsilon_j \eta_j     \qquad \qquad  (1 \leq j \leq 6).
\end{align*}
\end{proposition}
 \begin{proof} 
 Routine using Definitions  \ref{def:hvec}, \ref{def:Gmat} and Lemmas \ref{lem:Gp}, \ref{lem:HtGH}.
 \end{proof}
 
 \begin{lemma}
 \label{lem:orthogxh}
  For $2 \leq j \leq 6$ we have
  \begin{align*}
  \langle E{\hat x}, h_j \rangle =0, \qquad \qquad \langle E{\hat y}, h'_j \rangle =0.
  \end{align*}
  Moreover,
  \begin{align*}
  \langle E{\hat x}, h_1 \rangle =  \vert X \vert^{-1} \theta_1 \eta_1, \qquad \qquad   \langle E{\hat y}, h'_1 \rangle = \vert X \vert^{-1} \theta_1 \eta_1.
  \end{align*}
 \end{lemma}
 \begin{proof}
 By Lemma \ref{lem:h1}  and Proposition  \ref{prop:horthog}.
 \end{proof}

 \begin{lemma} \label{lem:gamma}
 We have
 \begin{align*}
 E{\hat y} = \sum_{j=1}^6 \gamma_j h_j, \qquad \qquad E{\hat x} = \sum_{j=1}^6 \gamma_j h'_j     
 \end{align*}
 where
 \begin{align*}
&\gamma_1 = \frac{1}{\theta_1}\,\frac{q^{N-k}-q^{N-D}-q^D+1}{(q^{N-D}-1)(q^D-1)}, \qquad \quad \gamma_2 = \frac{ q^{1-k}(q-1)}{(q^{N-D}-1)(q^D-1)}, \\
& \gamma_3 = \frac{ q^{1-k}(q-1)}{(q^{N-D}-1)(q^D-1)}, \qquad \qquad \qquad \gamma_4 = \frac{q^{1-k}}{q-1}, \\
& \gamma_5 = \frac{ q^{1-k}(q-1)}{(q^{N-D}-1)(q^D-1)}, \qquad \qquad \qquad
\gamma_6=0.
\end{align*}
 \end{lemma}
 \begin{proof} We  verify the first assertion.
  We define $Y= E{\hat y} - \sum_{j=1}^6 \gamma_j h_j$ and show $Y=0$.
 Using Lemmas \ref{lem:Osum}, \ref{lem:dep} and Definition \ref{def:hvec}, we express $Y$ as a linear combination of  $\lbrace E\widehat O_i \rbrace_{i=1}^6$.
 One checks that in this linear combination, each coefficient is 0. Therefore, $Y=0$ and the first assertion is verified.
 The second assertion is similarly verified.
 \end{proof}
 
 \begin{proposition} \label{prop:hmhp}
 For $1 \leq j \leq 6$ we have
 \begin{align}
 h_j - h'_j = \mu_j (E{\hat x} - E{\hat y}),     \label{eq:hmhp}
 \end{align}
 where
 \begin{align}
 \mu_j = \sum_{i=1}^6 \lambda_i H_{i,j}.       \label{eq:mu}
 \end{align}
 \end{proposition}
 \begin{proof} 
To verify \eqref{eq:hmhp}, in the left-hand side
eliminate $h_j$, $h'_j$ using Definition  \ref{def:hvec}, and evaluate
the result using Theorem  \ref{thm:bbalanced}.
\end{proof}
 
 \begin{lemma} \label{lem:moremu} The scalars $\lbrace \mu_j \rbrace_{j=1}^6$ from Proposition
 \ref{prop:hmhp} are given by
 \begin{align*}
& \mu_1 = \theta_1, \qquad \qquad 
 \mu_2 = -q^{N-D-1} \frac{q^D-q^k}{q-1}, \\
& \mu_3 = -q^{D-1} \frac{q^{N-D}-q^k}{q-1},
 \qquad \qquad \mu_4 = -(q-2) q^{k-1}, \\
 &\mu_5 = -q^{k-1} \frac{(q^{N-D}-q)(q^{D-1}-1)+q^{N-k}(q^{k-1}-1)(q-1)     }{(q-1)^2}, \qquad \quad \mu_6=0.
 \end{align*}
 \end{lemma}
 \begin{proof} Evaluate \eqref{eq:mu} using  Definition \ref{def:Hmatrix}
and
 Lemma    \ref{lem:lambdaValues}.
 \end{proof}

 \noindent The following result is for later use.
 \begin{lemma} \label{lem:mugam} We have
 \begin{align*}
  -1= \sum_{j=1}^6 \gamma_j \mu_j, \qquad \qquad 0 = \sum_{j=1}^6 \vartheta_j \gamma_j \mu_j.
  \end{align*}
  \end{lemma} 
  \begin{proof}
  By  Definition \ref{def:vth} and Lemmas \ref{lem:gamma}, \ref{lem:moremu}.
  \end{proof}

 \section{Decomposing $S$ into its symmetric and antisymmetric part}
 
 We continue to discuss the bilinear forms graph $\Gamma=(X, \mathcal R)$ and the fixed vertices $x,y \in X$ at distance $\partial(x,y)=k$.
 Recall the subspace $S=S(x,y) = S(y,x)$ of $EV$ from Definition \ref{def:N}  and Proposition \ref{prop:symS}. In this section,
 we introduce two subspaces of $S$, called the symmetric part and antisymmetric part. We show that $S$ is an orthogonal direct sum
  its symmetric and antisymmetric part.
 
 \begin{definition} \label{def:Ocheck} For $1 \leq i \leq 6$ we define the vector
 \begin{align*}
 O^\vee_i  = \widehat O_i - \lambda_i  \hat x,
 \end{align*}
 where $\lambda_i $ is from Theorem  \ref{thm:bbalanced}.  
 \end{definition}
 
 \begin{lemma} \label{lem:EOC} 
 For $1 \leq i \leq 6$,
 \begin{align*}
 E O^\vee_i = E \widehat O_i - \lambda_i  E \hat x = E \widehat O'_i - \lambda_i  E \hat y.
 \end{align*}
 \end{lemma}
 \begin{proof} The first equality is from Definition \ref{def:Ocheck}. The second equality
 is from  Theorem \ref{thm:bbalanced}.
 \end{proof}
 
 \begin{lemma} \label{lem:OCdep}
 We have
 \begin{align*}
 0 = \sum_{i=1}^6 E O^\vee_i.
 \end{align*}
 \end{lemma}
 \begin{proof}  By Lemmas \ref{lem:Osum}, \ref{lem:com} and Definition \ref{def:Ocheck}. 
 \end{proof} 
 
 \begin{lemma} \label{lem:xpy} We have
 \begin{align*}
    E\hat x + E \hat y =   q^{1-k} E O^\vee_1  - q^{1-k} E O^\vee_2.
 \end{align*}
 \end{lemma}
 \begin{proof} To verify this equation, eliminate $O^\vee_1$ and $ O^\vee_2$ using Definition \ref{def:Ocheck},
 and evaluate the result using Lemma \ref{lem:dep}.
 \end{proof}

  \noindent The following formulas are handy.
  \begin{lemma} \label{lem:handy}
  For $1 \leq i \leq 6$ we have
  \begin{align*}
  E \widehat O_i &= E O^\vee_i + \frac{\lambda_i}{2} \bigl( E\hat x + E \hat y\bigr) + \frac{\lambda_i}{2} \bigl( E\hat x - E \hat y\bigr),
  \\
  E \widehat O'_i &= E O^\vee_i + \frac{\lambda_i}{2} \bigl( E\hat x + E \hat y\bigr) - \frac{\lambda_i}{2} \bigl( E\hat x - E \hat y\bigr).
  \end{align*}
  \end{lemma}
  \begin{proof} By Lemma \ref{lem:EOC}.
  \end{proof}

 \begin{lemma} \label{lem:orthog} 
 For $1 \leq i \leq 6$,
 \begin{align*}
 \langle E O^\vee_i, E\hat x - E \hat y \rangle = 0.
 \end{align*}
 \end{lemma}
 \begin{proof}  By Lemma  \ref{lem:xyip} and the first equality in Lemma \ref{lem:EOC},
 \begin{align*}
\vert X \vert  \langle E O^\vee_i, E\hat x \rangle = \vert X \vert \langle E \widehat O_i - \lambda_i E \hat x, E \hat x \rangle = \vert O_i \vert \theta^*_1-\lambda_i \theta^*_0.
 \end{align*}
 \noindent  By Lemma  \ref{lem:xyip} and the second equality in Lemma \ref{lem:EOC},
  \begin{align*}
\vert X \vert  \langle E O^\vee_i, E\hat y \rangle = \vert X \vert \langle E \widehat O'_i - \lambda_i E \hat y, E \hat y \rangle = \vert O'_i \vert \theta^*_1-\lambda_i \theta^*_0.
 \end{align*}
 By construction $\vert O_i \vert = \vert O'_i\vert$. The result follows.
 \end{proof}

 \begin{definition} \label{def:symS} \rm Define the subspaces
  \begin{align*}
  {\rm Sym}(S) &= {\rm Span} \lbrace E O^\vee_i \vert 1 \leq i \leq 6\rbrace, \\
  {\rm ASym}(S) &= {\rm Span} \lbrace E \hat x - E \hat y \rbrace.
  \end{align*}
  We call ${\rm Sym}(S)$ (resp.  ${\rm ASym}(S)$) the {\it symmetric part} (resp. {\it antisymmetric part}) of $S$.
  \end{definition} 
  
  \begin{lemma} \label{lem:xpyS} We have
  \begin{align*}
  E \hat x + E \hat y \in {\rm Sym}(S).
  \end{align*}
  \end{lemma}
  \begin{proof} By Lemma \ref{lem:xpy} and Definition \ref{def:symS}.
  \end{proof}

  \begin{proposition} \label{prop:symDecomp} We have
  \begin{align*}
  S = {\rm Sym} (S) + {\rm ASym}(S) \qquad \quad \hbox{\rm (orthogonal direct sum).}
  \end{align*}
  The dimensions of ${\rm Sym} (S)$ and  ${\rm ASym}(S)$ are 5 and 1, respectively.
  \end{proposition}
\begin{proof}
We first show that  $S = {\rm Sym} (S) + {\rm ASym}(S)$. The inclusion $\subseteq $ follows from Definition  \ref{def:symS} and Lemmas \ref{lem:handy}, \ref{lem:xpyS}.
The inclusion $\supseteq$ follows from Corollary \ref{cor:xy}  and Definition \ref{def:Ocheck}. We have shown $S = {\rm Sym} (S) + {\rm ASym}(S)$. By Lemma \ref{lem:orthog},
the sum is orthogonal and hence direct. The vector $E\hat x - E\hat y$ is nonzero by Lemma  \ref{lem:xyind}, so ${\rm ASym}(S)$ has dimension 1. The dimension of
${\rm Sym}(S)$ is $6-1=5$.
\end{proof}

\begin{definition}\label{def:hcheck} \rm
For $1 \leq j \leq 6$ we define the vector
\begin{align*}
h^\vee_j= h_j - \mu_j E{\hat x} = h'_j - \mu_j E{\hat y},
\end{align*}
where $\mu_j$ is from  \eqref{eq:mu} and Lemma \ref{lem:moremu}.
\end{definition}

\begin{lemma} \label{lem:h1check} The following hold:
\begin{enumerate}
\item[\rm (i)]  $h^\vee_1=0$;
\item[\rm (ii)]  $h^\vee_6 = h_6 = h'_6$.
\end{enumerate}
\end{lemma}
\begin{proof} (i) By Lemmas \ref{lem:h1}, \ref{lem:moremu} and Definition \ref{def:hcheck},
\begin{align*}
h^\vee_1 = h_1 - \mu_1 E{\hat x} = \theta_1 E{\hat x} -\theta_1 E{\hat x} = 0.
\end{align*}
\noindent (ii) By Definition \ref{def:hcheck} and  $\mu_6=0$.
\end{proof}

\begin{lemma} \label{lem:hO}
For $1 \leq j \leq 6$,
\begin{align*}
h^\vee_j = \sum_{i=1}^6 H_{i,j} EO^\vee_i.
\end{align*}
\end{lemma} 
\begin{proof} By Definitions \ref{def:hvec}, \ref{def:hcheck} along with Lemma \ref{lem:EOC}   and  \eqref{eq:mu}.
\end{proof}

 \begin{lemma} \label{lem:xhy} We have
 \begin{align*}
 E{\hat x} + E{\hat y} = \sum_{j=2}^5 \gamma_j h^\vee_j,
 \end{align*}
 where the scalars $\gamma_j$ are from Lemma \ref{lem:gamma}.
 \end{lemma} 
 \begin{proof} Using Lemmas \ref{lem:gamma}, \ref{lem:mugam} and Definition \ref{def:hcheck},
 \begin{align*}
 E{\hat x} + E{\hat y} &= E{\hat x} + \sum_{j=1}^6 \gamma_j h_j 
 =   E{\hat x} + \sum_{j=1}^6 \gamma_j \bigl( h^\vee_j + \mu_j E{\hat x} \bigr) \\
 &= E{\hat x} \Biggl( 1 + \sum_{j=1}^6 \gamma_j \mu_j \Biggr) + \sum_{j=1}^6 \gamma_j h^\vee_j 
 = \sum_{j=1}^6 \gamma_j h^\vee_j.
 \end{align*}
 \noindent The result follows in view of   $h^\vee_1=0$   and $\gamma_6=0$.
 \end{proof}

\begin{lemma} \label{lem:hSAS} For $1 \leq j \leq 6$ we have
\begin{align*}
h_j &= h^\vee_j + \frac{\mu_j}{2} \bigl( E{\hat x} + E{\hat y}\bigr) + \frac{\mu_j}{2} \bigl( E{\hat x} - E{\hat y}\bigr), \\ 
h'_j &= h^\vee_j + \frac{\mu_j}{2} \bigl( E{\hat x} + E{\hat y}\bigr) - \frac{\mu_j}{2} \bigl( E{\hat x} - E{\hat y}\bigr).
\end{align*}
\end{lemma}
\begin{proof} By Definition \ref{def:hcheck}.
\end{proof} 

\begin{proposition}\label{lem:hSbasis} The vectors $\lbrace h^\vee_j \rbrace_{j=2}^6$ form a basis for ${\rm Sym}(S)$.
\end{proposition}
\begin{proof} By Definition \ref{def:symS}  and Lemma \ref{lem:hO}, the following holds for 
$2 \leq j \leq 6$:
\begin{align*}
h^\vee_j \in {\rm Span}\lbrace E O^\vee_i \vert 1 \leq i \leq 6 \rbrace = {\rm Sym}(S).
\end{align*}
By Proposition \ref{prop:symDecomp}, the dimension of ${\rm Sym}(S)$ is 5.
The vectors $\lbrace h^\vee_j \rbrace_{j=2}^6$ are linearly independent, because $\lbrace h_j \rbrace_{j=1}^6$ are linearly independent and
$h^\vee_j = h_j - \frac{\mu_j}{\mu_1} h_1$ for $2 \leq j \leq 6$. The result follows.
\end{proof}

 \section{The Norton algebra}
  We continue to discuss the bilinear forms graph $\Gamma=(X, \mathcal R)$.
  In this section, we bring in the Norton algebra structure on $EV$ \cite{norton, norton2}. We investigate the subspaces $S$, ${\rm Sym}(S)$, ${\rm ASym}(S)$ from
  the Norton algebra point of view.
  \medskip
  
  \noindent We recall the entry-wise product $\circ: V \times V \to V$. Pick $u,v \in V$ and write
  \begin{align*}
   u = \sum_{x \in X} u_x \hat x, \qquad  \qquad  v = \sum_{x \in X} v_x \hat x, \qquad \qquad u_x, v_x \in \mathbb R.
   \end{align*}
   Then
   \begin{align*}
   u \circ v = \sum_{x \in X} u_x v_x \hat x.
   \end{align*}
   
  \begin{definition}  \rm (See \cite[Proposition~5.2]{norton}.) 
 The {\it Norton algebra} $EV$ is the vector space $EV$ together with the  product 
    \begin{align*}
  u \star v = E(u \circ v) \qquad \qquad u, v \in EV.
  \end{align*} 
  \end{definition}
  
  \noindent The Norton algebra $EV$
  is commutative but not associative. See \cite{norton, norton2} for background information about the Norton algebra.
   
   \begin{lemma} \label{lem:uvw} Pick $u,v,w \in EV$ and write
   \begin{align*}
   u = \sum_{x \in X} u_x \hat x, \qquad   v = \sum_{x \in X} v_x \hat x, \qquad   w = \sum_{x \in X} w_x \hat x \qquad \quad u_x, v_x, w_x \in \mathbb R.
   \end{align*}
   Then 
   \begin{align*}
   \langle u \star v, w \rangle = \langle v \star w, u \rangle = \langle w \star u, v\rangle = \sum_{x \in X} u_x v_x w_x.
   \end{align*}
   \end{lemma}
   \begin{proof} Observe that
   \begin{align*}
  & \langle u \star v, w \rangle = \langle E(u \circ v), w\rangle = \langle u \circ v, E^t w \rangle \\
  &= \langle u \circ v, E w \rangle = \langle u \circ v, w \rangle = \sum_{x \in X} u_x v_x w_x.
   \end{align*}
   The remaining equalities are similarly obtained.
   \end{proof}

  \begin{lemma}  \label{lem:xsx} {\rm (See \cite[Lemma~3.2]{norton2}.)} 
  For $x \in X$,
  \begin{align*}
  E\hat x \star E \hat x = \vert X \vert^{-1} q^1_{1,1} E\hat x.
  \end{align*}
  \end{lemma}

  \begin{lemma} \label{lem:N1} {\rm (See \cite[Lemma~3.6 and Corollary~3.8]{norton2}.)} For adjacent $x,y \in X$,
  \begin{align*}
  E \hat x \star E \hat y = \vert X \vert^{-1} \sum_{z \in \Gamma(x) \cap \Gamma(y)} E \hat z.
  \end{align*}
  \end{lemma}
  
  \noindent For the rest of this section, we fix vertices $x,y \in X$ at distance $\partial(x,y)=k$,  as in Sections 6--10.
  Recall the subspace $S=S(x,y) = S(y,x)$ of $EV$ from Definition \ref{def:N}  and Proposition \ref{prop:symS}. 
   We will show that $E\hat x \star S \subseteq S$ and  $E\hat y \star S \subseteq S$. We conjecture that $S \star S \subseteq S$.

\medskip

\noindent
   Recall from Proposition \ref{prop:symDecomp} the direct sum decomposition
  $S = {\rm Sym} (S) + {\rm ASym}(S)$.
   We will show that 
   \begin{align*}
 {\rm ASym}(S) \star {\rm ASym}(S) &\subseteq {\rm Sym}(S), \\
 {\rm Sym}(S) \star {\rm ASym}(S) &\subseteq {\rm ASym}(S), \\
 (E{\hat x} + E{\hat y} ) \star {\rm Sym}(S) &\subseteq {\rm Sym}(S).
 \end{align*}
  
  \begin{proposition} \label{cor:Nadj}
  For $1 \leq j \leq 6$,
  \begin{align*}
  E\hat x \star E \widehat O_j = \vert X \vert^{-1} \sum_{i=1}^6 C_{i,j} E \widehat O_i.
  \end{align*}
  \end{proposition}
  \begin{proof} By the definition of $C$ in Proposition \ref{lem:Centry}, together with Lemma \ref{lem:N1}.
  \end{proof}
  
  \begin{lemma} \label{lem:Norton}
  We have
  \begin{align*}
  E\hat x \star S \subseteq S, \qquad \qquad E\hat y \star S \subseteq S.
  \end{align*}
  \end{lemma}
  \begin{proof} The inclusion on the left follows from Definition \ref{def:N} and Proposition \ref{cor:Nadj}. The inclusion
  on the right holds by Proposition  \ref{prop:symS}.
  \end{proof}
 
 \noindent By Corollary \ref{cor:xy} and Lemma \ref{lem:Norton}, we have $E \hat x \star E \hat y \in S$. In the next two
 results we give more detail.
 
 \begin{lemma}
 \label{lem:xStary} {\rm (See \cite[Theorem~3.7]{norton2}.)}
 We have
 \begin{align*}
 E \hat x \star E \hat y =
 \frac{(\theta^*_{k-1}-\theta^*_k) E \widehat O_1 + (\theta^*_{k+1}-\theta^*_k) E \widehat O_6 + (\theta_1-\theta_2)\theta^*_k E \hat x + (\theta_2 - \theta_0) E\hat y }{\vert X \vert (\theta_1 - \theta_2)}.
 \end{align*}
 \end{lemma} 
 \begin{proof} To verify this equation, on the left-hand side eliminate $E\hat y$ using Lemma  \ref{lem:dep}, and evaluate the result using Proposition \ref{cor:Nadj}.
 \end{proof}
 
  \begin{lemma}
 \label{lem:xStary2}  
 {\rm (See \cite[Theorem~4.4]{norton2}.)}
 We have
 \begin{align*}
 E \hat x \star E \hat y =
 \frac{(\theta^*_{k-1}-\theta^*_k) E O^\vee_1 + (\theta^*_{k+1}-\theta^*_k) E O^\vee_6 + (\theta_2 - \theta_0)(E\hat x + E\hat y) }{\vert X \vert (\theta_1 - \theta_2)}.
 \end{align*}
 \end{lemma} 
 \begin{proof} To verify this equation, eliminate $E O^\vee_1$ and $E O^\vee_6$ using the first equation in Lemma \ref{lem:EOC}, and evaluate the result using
 Lemma \ref{lem:xStary}.
 \end{proof}

\begin{corollary} \label{cor:xStary3}
We have
\begin{align*}
 E \hat x \star E \hat y \in {\rm Sym}(S).
 \end{align*}
 \end{corollary}
 \begin{proof}  By Definition \ref{def:symS} and Lemmas  \ref{lem:xpyS}, \ref{lem:xStary2}.
 \end{proof}
 
 \begin{proposition} \label{cor:xmy}
 We have
 \begin{align*}
 {\rm ASym}(S) \star {\rm ASym}(S) \subseteq {\rm Sym}(S).
 \end{align*}
 \end{proposition}
 \begin{proof} By Definition \ref{def:symS}, ${\rm ASym}(S) = {\rm Span}\lbrace E\hat x - E \hat y \rbrace$. By Lemmas  \ref{lem:xpyS},  \ref{lem:xsx} and Corollary \ref{cor:xStary3}, 
 \begin{align*} 
 \bigl( E\hat x - E \hat y\bigr) \star  \bigl( E\hat x - E \hat y\bigr) & = E \hat x \star E\hat x + E\hat y \star E\hat y - E\hat x \star E\hat y - E\hat y \star E\hat x \\
 &= \vert X \vert^{-1} q^1_{1,1} \bigl( E\hat x + E \hat y\bigr) -2 E\hat x \star E\hat y
      \\
      &  \in {\rm Sym}(S).
 \end{align*}
 The result follows.
 \end{proof}

 \begin{lemma} \label{lem:xOV} For $1 \leq j \leq 6$ we have
 \begin{align}
 E{\hat x} \star E O^\vee_j &= \vert X \vert^{-1} \sum_{i=1}^6 C_{i,j} E O^\vee_i 
 +  \frac{ \sum_{i=1}^6 \lambda_i C_{i,j}   -\lambda_j q^1_{1,1} }{\vert X \vert} E{\hat x}, \label{eq:xOV} \\
  E{\hat y} \star E O^\vee_j &= \vert X \vert^{-1} \sum_{i=1}^6 C_{i,j} E O^\vee_i 
 +  \frac{ \sum_{i=1}^6 \lambda_i C_{i,j}   -\lambda_j q^1_{1,1} }{\vert X \vert} E{\hat y}. \label{eq:yOV} 
 \end{align}
 \end{lemma}
 \begin{proof} We first obtain \eqref{eq:xOV}.
 By  Lemmas  \ref{lem:EOC}, \ref{lem:xsx}  and Proposition \ref{cor:Nadj},
 \begin{align*}
  E \hat x \star E O^\vee_j 
& = E \hat x \star \Bigl( E \widehat O_j - \lambda_j E\hat x\Bigr) \\
 &= \vert X \vert^{-1} \Biggl( \sum_{i=1}^6 C_{i,j} E \widehat O_i -\lambda_j q^1_{1,1} E\hat x         \Biggr) \\
  &= \vert X \vert^{-1} \Biggl( \sum_{i=1}^6 C_{i,j} \bigl( E O^\vee_i + \lambda_i E\hat x\bigr) -\lambda_j q^1_{1,1} E\hat x         \Biggr).
 \end{align*}
By these comments we obtain \eqref{eq:xOV}. The equation \eqref{eq:yOV} is similarly obtained.
 \end{proof}
 
 \begin{lemma} \label{lem:xydiff} For $1 \leq j \leq 6$,
 \begin{align*}
 E O^\vee_j \star \bigl( E\hat x - E \hat y\bigr)    =  \frac{ \sum_{i=1}^6 \lambda_i C_{i,j}   -\lambda_j q^1_{1,1} }{\vert X \vert}                \bigl( E\hat x - E\hat y\bigr).
 \end{align*}
 \end{lemma}
 \begin{proof}  By Lemma \ref{lem:xOV} and since $\star$ is commutative.
 \end{proof} 
 
 \begin{proposition}
 \label{prop:SymStarASym}
 We have
 \begin{align*}
 E O^\vee_1 \star \bigl( E\hat x - E \hat y\bigr) &= \frac{q^{k-1} (q^D+q^{N-D} -2q^{k-1}-q)}{\vert X \vert} \bigl( E\hat x - E \hat y\bigr),\\
  E O^\vee_2 \star \bigl( E\hat x - E \hat y\bigr) &= - \frac{2 q^{k-1} (q^{k-1}-1)}{\vert X \vert} \bigl( E\hat x - E \hat y\bigr),\\
   E O^\vee_3 \star \bigl( E\hat x - E \hat y\bigr) &= - \frac{(q-2)q^{k-1}(2q^{k-1}-1)}{\vert X \vert} \bigl( E\hat x - E \hat y\bigr),\\
    E O^\vee_4 \star \bigl( E\hat x - E \hat y\bigr) &= \frac{(q^D-2q^k)(q^{N-D}-q^k)}{q\vert X \vert} \bigl( E\hat x - E \hat y\bigr),\\
     E O^\vee_5 \star \bigl( E\hat x - E \hat y\bigr) &= \frac{(q^D-q^k)(q^{N-D}-2q^k)}{q \vert X \vert} \bigl( E\hat x - E \hat y\bigr),\\
      E O^\vee_6 \star \bigl( E\hat x - E \hat y\bigr) &= - \frac{2(q^D-q^k)(q^{N-D}-q^k)}{q\vert X \vert} \bigl( E\hat x - E \hat y\bigr).
 \end{align*}
 \end{proposition}
 \begin{proof} For the fraction in Lemma \ref{lem:xydiff}, evaluate the numerator 
 using Lemmas \ref{lem:Centry}, \ref{lem:lambdaValues} and \eqref{eq:q111}. The result follows.
 \end{proof}
 
 \noindent  \begin{proposition} \label{cor:SAS}
 We have
 \begin{align*}
  {\rm Sym}(S) \star {\rm ASym}(S) \subseteq {\rm ASym}(S).      
 \end{align*}
\end{proposition}
\begin{proof}  By Definition \ref{def:symS} and Proposition  \ref{prop:SymStarASym}.
\end{proof}
 
 \begin{lemma} \label{lem:xpySym} For $1 \leq j \leq 6$,
 \begin{align*}
 (E{\hat x} + E{\hat y} ) \star E O^\vee_j =  2 \vert X \vert^{-1} \sum_{i=1}^6 C_{i,j} E O^\vee_i 
 + \frac{ \sum_{i=1}^6 \lambda_i C_{i,j} - \lambda_j q^1_{1,1}}{\vert X \vert} ( E{\hat x} + E{\hat y}).
 \end{align*}
 \end{lemma}
 \begin{proof} By Lemma \ref{lem:xOV}.
 \end{proof}
 
 \begin{proposition} \label{prop:xpySS} We have
 \begin{align*}
 (E{\hat x} + E{\hat y}) \star {\rm Sym}(S) \subseteq {\rm Sym}(S).
 \end{align*}
 \end{proposition}
 \begin{proof}  By Definition \ref{def:symS} and Lemmas  \ref{lem:xpyS},  \ref{lem:xpySym}.
 \end{proof}
 
  \begin{conjecture} \label{prop:SSS} We have
 \begin{align*}
 {\rm Sym}(S) \star {\rm Sym}(S) \subseteq {\rm Sym}(S).
 \end{align*}
 \end{conjecture}

 \begin{proposition} \label{prop:hEig}
 For $1 \leq j \leq 6$,
 \begin{align*}
 E{\hat x} \star h_j = \vert X \vert^{-1} \vartheta_j h_j.       
 \end{align*}
 \end{proposition}
 \begin{proof} By Lemma \ref{lem:Hcol},  Definition \ref{def:hvec}, and Proposition \ref{cor:Nadj}.
 \end{proof}

 \begin{lemma} \label{lem:xyh} We have
 \begin{align*}
 E{\hat x} \star E{\hat y} = \vert X \vert^{-1} \sum_{j=1}^5 \gamma_j \vartheta_j h_j.
 \end{align*}
 \end{lemma}
 \begin{proof} By Lemma \ref{lem:gamma} and Proposition \ref{prop:hEig} along with $\gamma_6=0$.
\end{proof} 
 
  \begin{lemma} \label{lem:xyhv} We have
 \begin{align*}
 E{\hat x} \star E{\hat y} = \vert X \vert^{-1} \sum_{j=2}^5 \gamma_j \vartheta_j h^\vee_j.
 \end{align*}
 \end{lemma}
 \begin{proof} By Lemmas  \ref{lem:mugam}, \ref{lem:xyh} and Definition \ref{def:hcheck} along with $h^\vee_1=0$ and $\gamma_6=0$.
\end{proof} 

 \begin{lemma} \label{lem:xyacth} For $1 \leq j \leq 6$ we have
 \begin{align*}
 E{\hat x} \star h^\vee_j &=\vert X \vert^{-1} \vartheta_j h^\vee_j + \vert X \vert^{-1} (\vartheta_j - \vartheta_1)\mu_j E{\hat x},\\
 E{\hat y} \star h^\vee_j &=\vert X \vert^{-1} \vartheta_j h^\vee_j + \vert X \vert^{-1} (\vartheta_j - \vartheta_1) \mu_j E{\hat y}.
 \end{align*}
 \end{lemma}
 \begin{proof} The first assertion follows from Definition \ref{def:hcheck}, Lemma  \ref{lem:xsx} along with Proposition \ref{prop:hEig} and  $\vartheta_1=a_1=q^1_{1,1}$.
 The second assertion is similarly obtained.
 \end{proof}
 
 \begin{proposition} \label{prop:xydiff}  For $1 \leq j \leq 6$,
 \begin{align*}
 h^\vee_j \star \bigl( E {\hat x} - E{\hat y} \bigr) = \vert X \vert^{-1}(\vartheta_j - \vartheta_1)\mu_j \bigl( E {\hat x} - E{\hat y} \bigr).
 \end{align*}
 \end{proposition}
 \begin{proof} By Lemma \ref{lem:xyacth} and since $\star$ is commutative.
 \end{proof}
 
  \begin{proposition} \label{lem:xyactSym} For $1 \leq j \leq 6$,
 \begin{align*}
\bigl( E{\hat x} +E{\hat y} \bigr) \star h^\vee_j =2\vert X \vert^{-1} \vartheta_j h^\vee_j + \vert X \vert^{-1} (\vartheta_j - \vartheta_1)\mu_j \bigl( E{\hat x}+ E{\hat y}\bigr).
 \end{align*}
 \end{proposition}
 \begin{proof} By Lemma \ref{lem:xyacth}.
 \end{proof}
 

 \section{More about the Norton algebra}
  We continue to discuss the bilinear forms graph $\Gamma=(X, \mathcal R)$ and the fixed vertices $x,y \in X$ at distance $\partial(x,y)=k$.
  Recall the  subspace
  $S =S(x,y)=S(y,x)$ of $EV$ from Definition \ref{def:N} and Proposition  \ref{prop:symS}. 
  In this section, we investigate the vector
  $h^\vee_6 = h_6 = h'_6 \in S$ from Lemma  \ref{lem:h1check}(ii). For notational convenience, we denote this vector by $\omega$.
  We decompose $S$ into an orthogonal direct sum $S={\rm Span}\lbrace \omega\rbrace+\omega^\perp$, where $\omega^\perp$ denotes the
  orthogonal complement of $\omega$ in $S$. We describe the above decomposition from a Norton algebra point of view.
  We also express the strengthened balanced set condition Theorem   \ref{thm:bbalanced} using a Norton algebra point of view.
  
     \begin{definition} \label{def:omega} \rm
   Let $\omega = \omega(x,y)$ denote the vector $h^\vee_6 = h_6 = h'_6$ from Lemma  \ref{lem:h1check}(ii).
  Define $\omega_i = H_{i,6}$ for $1 \leq i \leq 6$, where $H$ is from Definition \ref{def:Hmatrix}.
  By construction,
   \begin{align}
   \omega = \sum_{i=1}^6 \omega_i E \widehat O_i     \label{eq:omega}
   \end{align}
  and
    \begin{align*}
 \omega_1 &=   \frac{ (q^D-q^k)(q^{N-D}-q^k)(q^{k-1}-1)}{q^k(q-1)}, \qquad \quad 
     \omega_2 = \frac{ (q^D-q^k)(q^{N-D}-q^k)}{q(q-1)}, \\
       \omega_3 &= \frac{ (q^D-q^k)(q^{N-D}-q^k)(q^{k-1}-1)}{q^k(q-1)}, \qquad \quad 
       \omega_4 =- \frac{ (q^D-q^k)(q^{k-1}-1)}{q-1}, \\
      \omega_5&= - \frac{ (q^{N-D}-q^k)(q^{k-1}-1)}{q-1}, \qquad \qquad \qquad
      \omega_6 =\frac{ (q^k-1)(q^{k-1}-1)}{q-1}.
  \end{align*}
\noindent Note that $0 \not=\omega \in S$.
   \end{definition}


\noindent Recall the vectors $\lbrace O^\vee_i \rbrace_{i=1}^6$ from Definition \ref{def:Ocheck}.
\begin{lemma} \label{lem:omegaSym} We have
\begin{align*}
\omega = \sum_{i=1}^6 \omega_i E O^\vee_i.
\end{align*}   
\end{lemma}
\begin{proof}  Apply Lemma \ref{lem:hO} with $j=6$.
\end{proof}   

\begin{lemma} \label{lem:omegaSym2} We have $\omega \in {\rm Sym}(S)$. In other words,
\begin{align*} 
\omega(x,y) = \omega(y,x).
\end{align*}
\end{lemma}
\begin{proof} The result follows from $\omega = h_6 = h'_6$. The result also follows from Lemmas \ref{lem:EOC},  \ref{lem:omegaSym} and Definition \ref{def:symS}.
\end{proof}    
  
  \begin{proposition} \label{prop:omega} We have
  \begin{align*}
  E\hat x \star \omega = -q \vert X \vert^{-1} \omega, \qquad \qquad
  E \hat y \star \omega = - q \vert X \vert^{-1} \omega.
  \end{align*}
 \end{proposition}
 \begin{proof} By Proposition  \ref{prop:hEig} (with $j=6$) along with Lemma \ref{lem:omegaSym2} and $\vartheta_6=-q$.
  \end{proof}

 \noindent Recall that $\omega \in S$.
\begin{definition} \label{def:omegap} \rm Define the subspace
\begin{align*}
\omega^\perp = \lbrace v \in S \vert \langle \omega, v \rangle =0 \rbrace.
\end{align*}
\end{definition}

\noindent  By  linear algebra,
\begin{align}
&S = {\rm Span}\lbrace \omega\rbrace + \omega^\perp \qquad \qquad \hbox{\rm (orthogonal direct sum).} \label{eq:Somomp}
\end{align}
\noindent Note that ${\rm dim}\,\omega^\perp = 5$. 

\begin{lemma}
\label{lem:wpbasis}
Each of the following is an orthogonal basis for $\omega^\perp$:
\begin{align*}
\lbrace h_j \rbrace_{j=1}^5, \qquad \qquad \lbrace h'_j \rbrace_{j=1}^5.
\end{align*}
Moreover,
\begin{align}
E{\hat x}  \in \omega^\perp, \qquad \qquad E{\hat y} \in \omega^\perp. \label{eq:xywp}
\end{align}
\end{lemma}
 \begin{proof} By Lemmas \ref{lem:hbasis}, \ref{lem:h1} along with Proposition \ref{prop:horthog} and Definition \ref{def:omega}.
  \end{proof}
 
\begin{lemma} \label{lem:actionxy} We have
\begin{align*}
E{\hat x} \star \omega^\perp \subseteq \omega^\perp, \qquad \qquad E{\hat y} \star \omega^\perp \subseteq \omega^\perp.
\end{align*}
\end{lemma}
\begin{proof} We first verify the inclusion on the left.  For $u \in \omega^\perp$ we show that $E{\hat x} \star u \in \omega^\perp$.
By construction $u \in \omega^\perp \subseteq S$, so $E{\hat x} \star u \in S$ by Lemma  \ref{lem:Norton}.
By Lemma \ref{lem:uvw} and Proposition  \ref{prop:omega},
\begin{align*}
\langle E{\hat x} \star u, \omega\rangle = \langle u, E{\hat x} \star \omega \rangle = -q \vert X \vert^{-1} \langle u, \omega\rangle = 0.
\end{align*}
By these comments $E{\hat x} \star u \in \omega^\perp$. We have verified the inclusion on the left. The inclusion on the right follows in view of  Lemma \ref{lem:omegaSym2}.
\end{proof}

\begin{proposition} \label{wpGen}
The subspace $\omega^\perp$ has a basis
\begin{align*}
&E{\hat y}, \qquad \quad E{\hat x} \star E{\hat y}, \qquad \quad E{\hat x} \star ( E{\hat x} \star E{\hat y}), \qquad \quad  E{\hat x} \star (E{\hat x} \star ( E{\hat x} \star E{\hat y})), \\
& E{\hat x} \star ( E{\hat x} \star (E{\hat x} \star ( E{\hat x} \star E{\hat y})))
\end{align*}
and another basis 
\begin{align*}
&E{\hat x}, \qquad \quad E{\hat y} \star E{\hat x}, \qquad \quad E{\hat y} \star ( E{\hat y} \star E{\hat x}), \qquad \quad  E{\hat y} \star (E{\hat y} \star ( E{\hat y} \star E{\hat x})), \\
& E{\hat y} \star ( E{\hat y} \star (E{\hat y} \star ( E{\hat y} \star E{\hat x}))).
\end{align*}
\end{proposition}
\begin{proof}  We verify the first assertion. 
By Lemma \ref{lem:wpbasis} the vectors $\lbrace h_j \rbrace_{j=1}^5$ form a basis for $\omega^\perp$.
By Lemma \ref{lem:actionxy},
there exists a $\mathbb R$-linear map
\begin{align*}
\mathcal E: \quad \begin{split} \omega^\perp &\to \omega^\perp \\
                                                            u &\mapsto  E{\hat x} \star u
                                                           \end{split}
                                                           \end{align*}                                                      
 The proposed basis for $\omega^\perp$ consists of the vectors $\lbrace \mathcal E^\ell  E{\hat y} \rbrace_{\ell=0}^4$.                                                         
By Lemma \ref{lem:gamma} and $\gamma_6=0$,
\begin{align}
E{\hat y} = \sum_{j=1}^5 \gamma_j h_j.       \label{eq:ysumh}
\end{align}
Inspecting the data in Lemma \ref{lem:gamma}, we have $\gamma_j \not=0$ for $1 \leq j \leq 5$. 
By \eqref{eq:ysumh} and Proposition \ref{prop:hEig},
\begin{align*}
      \mathcal E^\ell  E{\hat y} = \vert X \vert^{-\ell} \sum_{j=1}^5 \vartheta_j^\ell  \gamma_j h_j \qquad \qquad (0 \leq \ell \leq 4).
\end{align*}
 By Definition  \ref{def:vth}, the scalars
$\lbrace \vartheta_j \rbrace_{j=1}^5$ are mutually distinct. By these comments and linear algebra, the
vectors $\lbrace \mathcal E^\ell  E{\hat y} \rbrace_{\ell=0}^4$ form a basis for $\omega^\perp$.
We have verified the first assertion. The second assertion follows in view of  Lemma \ref{lem:omegaSym2}.
\end{proof}

\begin{proposition} \label{eq:spanwperp}
The subspace $\omega^\perp$ is spanned by the vectors 
\begin{align*}
& E{\hat z}_1\star ( E{\hat z_2} \star (E{\hat z_3} \star \cdots \star ( E{\hat z_{n-1}} \star E{\hat z_n} ) \cdots )), \\
&  n\geq 1, \qquad \quad z_i \in \lbrace x,y \rbrace \quad (1 \leq i \leq n).
\end{align*}
\end{proposition}
\begin{proof}  By   \eqref{eq:xywp} and Lemma \ref{lem:actionxy}, the given vectors are contained in $\omega^\perp$.
By Proposition  \ref{wpGen}, the given vectors span $\omega^\perp$.
\end{proof}

\begin{conjecture}\rm The subspace $\omega^\perp$ is the subalgebra of the Norton algebra $EV$ generated by $E{\hat x}$, $E{\hat y}$.
\end{conjecture}

\noindent In Proposition \ref{prop:symDecomp} and \eqref{eq:Somomp} we decomposed $S$ into an orthogonal direct sum of two subspaces.
Next, we consider how these decompositions are related.

\begin{lemma} \label{lem:3facts} The following sums are orthogonal and direct:
\begin{align*}
S &= {\rm Span}\lbrace \omega\rbrace + {\rm Span} \lbrace E{\hat x}- E{\hat y}\rbrace + \omega^\perp \cap {\rm Sym}(S), \\
{\rm Sym}(S) &= {\rm Span}\lbrace \omega\rbrace +  \omega^\perp \cap {\rm Sym}(S),\\
\omega^\perp &= {\rm Span} \lbrace E{\hat x} - E{\hat y}\rbrace + \omega^\perp \cap {\rm Sym}(S).
\end{align*}
\end{lemma}
\begin{proof} By Definition  \ref{def:symS}, Proposition \ref{prop:symDecomp}, Lemma \ref{lem:omegaSym2}, and \eqref{eq:Somomp}.
\end{proof}

\begin{lemma} \label{lem:xpys} The vectors $\lbrace h^\vee_j \rbrace_{j=2}^5$ form a basis for $ \omega^\perp \cap {\rm Sym}(S)$.
\end{lemma}
\begin{proof} By  Proposition \ref{lem:hSbasis}  and Lemma \ref{lem:wpbasis}.
\end{proof}

\begin{lemma} \label{lem:xpyAction} We have
\begin{align*}
E{\hat x} + E{\hat y} \in \omega^\perp \cap {\rm Sym}(S).
\end{align*}
Moreover,
\begin{align*}
(E{\hat x} + E{\hat y}) \star \bigl(  \omega^\perp \cap {\rm Sym}(S)   \bigr) \subseteq \omega^\perp \cap {\rm Sym}(S).
\end{align*}
\end{lemma}
\begin{proof} 
The first assertion follows from Lemma \ref{lem:xpyS} and  \eqref{eq:xywp}.
The second assertion follows from Proposition  \ref{prop:xpySS} and Lemma \ref{lem:actionxy}.
\end{proof}

\begin{proposition} \label{prop:BBBbasis}  The subspace  $\omega^\perp \cap {\rm Sym}(S)$     has a basis
\begin{align}
B, \qquad \quad B \star B, \qquad \quad B \star (B \star B), \qquad \quad B \star ( B \star (B \star B)),       \label{eq:4vec}
\end{align}
where $B = E{\hat x} + E{\hat y}$.
\end{proposition}
\begin{proof} 
Define the vectors
\begin{align*}
B_\ell = \sum_{j=2}^5 \vartheta^\ell_j \gamma_j h^\vee_j \qquad \qquad (0 \leq \ell \leq 3).
\end{align*}
By Lemma \ref{lem:xpys} and the construction, the vectors $\lbrace B_\ell \rbrace_{\ell=0}^3$ form a basis for $\omega^\perp \cap {\rm Sym}(S)$.
By Lemma  \ref{lem:xhy} we have $B_0=B$.
By Proposition  \ref{lem:xyactSym} the following holds for $0 \leq \ell \leq 2$:
\begin{align}
\vert X \vert B \star B_\ell = 2 B_{\ell+1} + B \sum_{j=2}^5 (\vartheta_j - \vartheta_1) \vartheta^\ell_j \gamma_j \mu_j.
\label{eq:B0Bt}
\end{align}
By \eqref{eq:B0Bt}, the vectors  $\lbrace B_\ell \rbrace_{\ell=0}^3$  and the vectors \eqref{eq:4vec} have the same
span. Consequently, the vectors \eqref{eq:4vec} form a basis for $\omega^\perp \cap {\rm Sym}(S)$.
\end{proof}

\begin{proposition} \label{eq:spanBBB}
The subspace $\omega^\perp \cap {\rm Sym}(S)$ is spanned by the vectors 
\begin{align*}
B, \qquad \quad B \star B, \qquad \quad B \star (B \star B), \qquad \quad B \star ( B \star (B \star B)),  \qquad    \ldots
\end{align*}
where $B = E{\hat x} + E{\hat y}$.
\end{proposition}
\begin{proof}  By Lemma \ref{lem:xpyAction},
 the given vectors are contained in $\omega^\perp \cap {\rm Sym}(S)$.
By Proposition  \ref{prop:BBBbasis}, the given vectors span $\omega^\perp \cap {\rm Sym}(S)$.
\end{proof}

\begin{conjecture} \label{conj:oneGen} \rm The subspace $\omega^\perp \cap {\rm Sym}(S)$ is the subalgebra of the Norton algebra $EV$ generated by
$E{\hat x} + E{\hat y}$.
\end{conjecture}

\noindent Recall the strengthened balanced set condition in Theorem  \ref{thm:bbalanced}. In our final result, 
we express this condition using a Norton algebra point of view.

\begin{theorem} \label{thm:main2} For an integer $n\geq 1$ and $1 \leq i \leq n$ let $z_i, {\sf z}_i$ denote a permutation of
$x, y$. 
Then
\begin{align*}
& E{\hat z}_1\star ( E{\hat z_2} \star (E{\hat z_3} \star \cdots \star ( E{\hat z_{n-1}} \star E{\hat z_n} ) \cdots )) \\
&-
E{\hat {\sf z}}_1\star ( E{\hat {\sf z}_2} \star (E{\hat {\sf z}_3} \star \cdots \star ( E{\hat {\sf z}_{n-1}} \star E{\hat {\sf z}_n} ) \cdots )) \\
& \in {\rm Span} \lbrace E{\hat x} - E{\hat y}\rbrace.
\end{align*}
\end{theorem}
\begin{proof}
Define
\begin{align*}
u &= E{\hat z}_1\star ( E{\hat z_2} \star (E{\hat z_3} \star \cdots \star ( E{\hat z_{n-1}} \star E{\hat z_n} ) \cdots )), \\
{\sf u} &=E{\hat {\sf z}}_1\star ( E{\hat {\sf z}_2} \star (E{\hat {\sf z}_3} \star \cdots \star ( E{\hat {\sf z}_{n-1}} \star E{\hat {\sf z}_n} ) \cdots )).
\end{align*}
The vector ${\sf u}$ is obtained from $u$ by exchanging the roles of $x,y$.
 We have $u, {\sf u} \in \omega^\perp$ by Proposition \ref{eq:spanwperp}.       
By the last equation in Lemma \ref{lem:3facts},
there exists $\alpha \in \mathbb R$ such that
\begin{align*}
u - \alpha (E{\hat x} - E{\hat y}) \in  \omega^\perp \cap {\rm Sym}(S).
\end{align*}
Exchanging the roles of $x,y$ we obtain
\begin{align*}
u - \alpha (E{\hat x} - E{\hat y})={\sf u} - \alpha (E{\hat y} - E{\hat x}).
 \end{align*}
 Therefore,
 \begin{align*}
 u - {\sf u}  = 2 \alpha (E{\hat x} - E{\hat y}) \in {\rm Span}\lbrace E{\hat x} - E{\hat y}\rbrace.
 \end{align*}
 \end{proof}

\begin{definition}\rm The condition in Theorem \ref{thm:main2} will be called the {\it bbalanced set condition}.
\end{definition}

\begin{remark}\rm It seems that the bbalanced set condition is unrelated to the Norton-balanced condition described in   \cite{nortonBalanced}.
\end{remark}

 \section{Directions for future research}
 
 \noindent In this section, we give some suggestions for future research.
 \medskip

\begin{problem}\label{prob:one} \rm
Let $\Gamma=(X,\mathcal R)$ denote a distance-regular graph with diameter $D\geq 3$. Let $E$ denote
a primitive idempotent for $\Gamma$ that is $Q$-polynomial in the sense of \cite[p.~251]{bbit}.
Investigate the case in which the bbalanced set condition shown in
Theorem \ref{thm:main2}
holds for every pair of distinct vertices $x,y \in X$.
If possible, classify the pairs $\Gamma, E$ that have this feature. To get started,
find examples of pairs $\Gamma, E$ that have this feature. Many candidates can be
found in \cite[Section~6.4]{bbit}.
\end{problem}

  \begin{problem}\rm Compute the nucleus \cite[Definition~6.8]{TerNucleus} of the bilinear forms graph. See  \cite{nucleusBOHOU, nucleus, seongNuc} for related work.
 \end{problem}

 \begin{problem}\rm For the bilinear forms graph, compute the fundamental module for the $S_3$-symmetric tridiagonal algebra \cite[Section~9]{vvv}. See \cite{sl4,vvv2, ter2hom} for related work.
 \end{problem}

 

\section{Acknowledgement} 
\noindent 
This paper was written during J. Williford's sabbatical visit to U. Wisconsin-Madison
(8/17/2025--12/14/2025). During this visit,
J. Williford was supported in part by the Simons Foundation Collaboration Grant 711898.


\bigskip


\noindent Paul Terwilliger \hfil\break
\noindent Department of Mathematics \hfil\break
\noindent University of Wisconsin \hfil\break
\noindent 480 Lincoln Drive \hfil\break
\noindent Madison, WI 53706-1388 USA \hfil\break
\noindent email: {\tt terwilli@math.wisc.edu }\hfil\break

\noindent Jason Williford \hfil\break
\noindent  Department of Mathematics and Statistics \hfil\break
\noindent University of Wyoming \hfil\break
\noindent  1000 E. University Ave. \hfil\break
\noindent  Laramie, WY 82071  USA \hfil\break
\noindent Email: {\tt jwillif1@uwyo.edu}\hfil\break

\section{Statements and Declarations}

\noindent {\bf Funding}: The author P. Terwilliger declares that no funds, grants, or other support were received during the preparation of this manuscript. 
The author J. Williford was supported in part by Simons Foundation Collaboration Grant 711898  during the preparation of this manuscript.
\medskip

\noindent  {\bf Competing interests}:  The authors  have no relevant financial or non-financial interests to disclose.
\medskip

\noindent {\bf Data availability}: All data generated or analyzed during this study are included in this published article.

\end{document}